\documentclass[11pt,a4paper]{amsart}
\usepackage[T1]{fontenc}
\usepackage[utf8]{inputenc}
\usepackage[english]{babel}
\usepackage{lmodern}
\usepackage[margin=1in]{geometry}
\usepackage{amssymb,mathtools}
\usepackage{cite}
\usepackage{microtype}
\usepackage{enumitem}
\usepackage{tikz}
\usetikzlibrary{arrows.meta,calc}
\usepackage[hidelinks,unicode]{hyperref}
\usepackage{bookmark}

\newtheorem{theorem}{Theorem}[section]
\newtheorem{lemma}[theorem]{Lemma}
\newtheorem{proposition}[theorem]{Proposition}
\newtheorem{corollary}[theorem]{Corollary}
\theoremstyle{definition}

\newtheorem{example}[theorem]{Example}
\newtheorem{problem}[theorem]{Problem}
\theoremstyle{remark}
\newtheorem{remark}[theorem]{Remark}
\numberwithin{equation}{section}

\newcommand{\R}{\mathbb R}
\newcommand{\Comp}{\operatorname{Comp}}
\newcommand{\col}{\operatorname{col}}
\newcommand{\supp}{\operatorname{supp}}
\newcommand{\MMS}{\operatorname{MMS}}
\newcommand{\mms}{\operatorname{mms}}
\newcommand{\poly}{\operatorname{poly}}
\DeclareMathOperator*{\argmin}{arg\,min}
\newcommand{\CC}[1]{\mathcal C_{#1}}
\newcommand{\inn}[1]{E_Q(#1)}
\newcommand{\bd}[1]{\partial_Q #1}
\newcommand{\ind}[1]{\mathbf 1[#1]}
\newcommand{\doi}[1]{\href{https://doi.org/#1}{\texttt{doi:}\nolinkurl{#1}}}
\setlist[enumerate]{label=\textup{(\roman*)},leftmargin=2.1em,itemsep=2pt}
\newcommand{\parhead}[1]{\par\smallskip\noindent\textit{#1}\enspace}
\title[Rainbow partitions of trees]{A rainbow partition theorem for trees\texorpdfstring{\\}{ }and connected maximin share allocations of chores}

\author[M. Anholcer]{Marcin Anholcer}
\address{Institute of Informatics and Quantitative Economics, Pozna\'n University of Economics and Business, Pozna\'n, Poland}
\email{marcin.anholcer@ue.poznan.pl}

\author[M. Bartkowiak]{Maciej Bartkowiak}
\address{Doctoral School, Pozna\'n University of Economics and Business, Pozna\'n, Poland}
\email{maciej.bartkowiak@ue.poznan.pl}

\author[B. Bosek]{Bart\l omiej Bosek}
\address{Institute of Theoretical Computer Science, Faculty of Mathematics and Computer Science, Jagiellonian University, Krak\'ow, Poland}
\email{bartlomiej.bosek@uj.edu.pl}

\author[J. Grytczuk]{Jaros\l aw Grytczuk}
\address{Faculty of Mathematics and Information Science, Warsaw University of Technology, Warsaw, Poland}
\email{jaroslaw.grytczuk@pw.edu.pl}

\author[Z. Lonc]{Zbigniew Lonc}
\address{Faculty of Mathematics and Information Science, Warsaw University of Technology, Warsaw, Poland}
\email{zbigniew.lonc@pw.edu.pl}

\author[P. Rz\k{a}\.zewski]{Paweł Rz\k{a}\.zewski}
\address{Faculty of Mathematics and Information Science, Warsaw University of Technology, Warsaw, Poland}
\email{pawel.rzazewski@pw.edu.pl}

\thanks{This research was funded by the National Science Centre, Poland (grant number: 2023/51/B/HS4/00829). The last author was 
supported by the European Research Council (ERC) under the European Union’s Horizon 2020 research and innovation programme grant agreement number 948057.}
\date{}
\hypersetup{pdftitle={A rainbow partition theorem for trees and connected maximin share allocations of chores},pdfauthor={Marcin Anholcer, Maciej Bartkowiak, Bartłomiej Bosek, Jarosław Grytczuk},pdfsubject={Connected partitions of trees and fair allocation of indivisible chores},pdfkeywords={tree partition, rainbow partition, connected fair division, indivisible chores, maximin share, colourful KKM theorem}}

\begin{document}
\begin{abstract}
Xiao, Qiu, and Huang (AAMAS 2023) and independently Lonc (personal communication) asked whether indivisible chores located at the vertices of a tree can always be allocated to $n$ agents in connected bundles so that the cost of every agent is at most its connected maximin share; for goods, this is a theorem of Bouveret, Cechl\'arov\'a, Elkind, Igarashi, and Peters. We answer the question affirmatively, even for monotone costs. The answer follows from a combinatorial theorem: if $\mathcal P_1,\ldots,\mathcal P_k$ are partitions of the vertex set of a finite tree, each into at most $k$ connected parts, then the vertex set can be split into disjoint sets $B_1,\ldots,B_k$, some possibly empty, such that each nonempty $B_i$ is connected and contained in a part of $\mathcal P_i$. Equivalently, if each of $k$ colours occurs on at most $k-1$ edges of a tree, then the vertices can be partitioned into connected sets labelled by distinct colours, none containing an edge of its own colour; in particular, one can choose for every colour a component of the forest obtained by deleting that colour so that the chosen components cover the tree. The bound is already best possible for paths, and for additive costs, the theorem is equivalent to the fair-division statement. The proof reduces the problem to inward partitions of oriented trees, which we obtain from the colourful KKM theorem on a simplex of edge weights, using a leaf-elimination labelling that remains compatible when weights vanish. We also give an algorithm running in time $k^{O(k)}$ plus polynomial time.


\end{abstract}
\maketitle

\section{Introduction}
\label{sec:intro}

Fair division of indivisible items under connectivity constraints models situations in which the items are arranged in a network, and every agent must receive a connected bundle: plots of land, offices along a corridor, or segments of a road network to be maintained by different contractors. In the model introduced by Bouveret, Cechl\'arov\'a, Elkind, Igarashi, and Peters~\cite{BouveretEtAl2017}, the items are the vertices of a graph, and every bundle must induce a connected subgraph; see~\cite{Suksompong2021} for a survey of this and other constraints in fair division. A central fairness benchmark is the maximin share (MMS) of Budish~\cite{Budish2011}, adapted to connectivity constraints in~\cite{BouveretEtAl2017}: the best guarantee that an agent can secure by partitioning the items into as many admissible (here, connected) bundles as there are agents and receiving the worst of them. For goods placed on a tree, Bouveret et~al.~\cite{BouveretEtAl2017} proved that an allocation into connected bundles giving every agent its maximin share always exists and can be computed in polynomial time. On cycles, such allocations need not exist~\cite{BouveretEtAl2017,LoncTruszczynski2020}.

For chores, that is, for items that impose costs, the situation on trees has been much less clear. If agent $i$ has cost function $\kappa_i$ and there are $n$ agents, the connected maximin share of agent $i$ is
\begin{equation}
 \MMS_i(T,n)=\min_{\mathcal P}\ \max_{P\in\mathcal P}\ \kappa_i(P),
 \label{eq:intro-mms}
\end{equation}
where $\mathcal P$ ranges over the partitions of the vertex set of the tree $T$ into at most $n$ connected parts. This is the smallest worst-bundle cost that the agent can guarantee by dividing the chores into feasible bundles itself. Xiao, Qiu, and Huang~\cite{XiaoEtAl2023,XiaoEtAl2023arXiv} observed that the last-diminisher argument behind the result for goods breaks down for chores: in its chore version, an agent enlarges the current bundle when it considers the bundle too light, and after such an enlargement, the rest of the tree may be disconnected~\cite[Section~3]{XiaoEtAl2023arXiv}. They introduced a group-satisfied method, in which several bundles are assigned to several agents at once, and proved that connected MMS allocations of chores exist, and can be computed in polynomial time, on trees of radius at most two (trees of depth at most three in their convention, in which a star has depth two) and on spiders, that is, on trees with exactly one vertex of degree at least three~\cite[Sections~5 and~6]{XiaoEtAl2023arXiv}. They wrote that ``it remains open whether MMS allocations of chores on trees always exist or not, which is a simple but annoying problem in chores allocation''~\cite[Abstract]{XiaoEtAl2023arXiv}, and in the concluding section of their paper they state their belief that such allocations always exist~\cite[Section~8]{XiaoEtAl2023arXiv}.

We confirm this belief. Connected MMS allocations of chores exist on every tree, for arbitrary monotone costs, and every agent can even be given a connected subset of one of its own benchmark parts (Theorem~\ref{thm:B}). This follows from a purely combinatorial theorem on partitions of trees, which is the main result of the paper.

\subsection{Main results}
\label{sec:results}

All graphs are finite. A set of vertices of a graph is \emph{connected} if it is nonempty and induces a connected subgraph; in particular, the empty set is not connected. A \emph{connected partition} of the vertex set is a partition into connected sets; as usual, the parts of a partition are nonempty. By contrast, an \emph{allocation} of a finite set $X$ to an index set $I$ is a family $(B_i)_{i\in I}$ of pairwise disjoint subsets of $X$, some of which may be empty, whose union is $X$. We write $[k]=\{1,\ldots,k\}$. For a graph $G$ and $F\subseteq E(G)$, we write $G-F$ for the graph obtained from $G$ by deleting the edges of $F$ but keeping all vertices, and $G[X]$ for the subgraph induced by $X\subseteq V(G)$; components of a graph are identified with their vertex sets.

\begin{theorem}[Rainbow partition theorem]
\label{thm:A}
Let $T$ be a finite tree, let $k\ge1$, and let $\mathcal P_1,\ldots,\mathcal P_k$ be connected partitions of $V(T)$ with $|\mathcal P_i|\le k$ for every $i\in[k]$. Then there is an allocation $(B_i)_{i\in[k]}$ of $V(T)$ such that every nonempty $B_i$ is connected and contained in a member of $\mathcal P_i$. If $|V(T)|\ge k$, the sets $B_i$ can be chosen nonempty.
\end{theorem}

We call such an allocation a \emph{rainbow partition subordinate to} $(\mathcal P_i)_{i\in[k]}$, and the sets $B_i$ its \emph{pieces}. Despite the name, some pieces may be empty; the nonempty pieces form a connected partition of $V(T)$. Each partition contributes one piece, and the piece of index $i$ must lie inside a part of its own partition $\mathcal P_i$, but need not respect the other partitions. Figure~\ref{fig:spider} shows an example. Theorem~\ref{thm:A} is not a statement about set systems alone: if the tree is replaced by a four-cycle, it fails already for $k=2$ (Section~\ref{sec:why}).

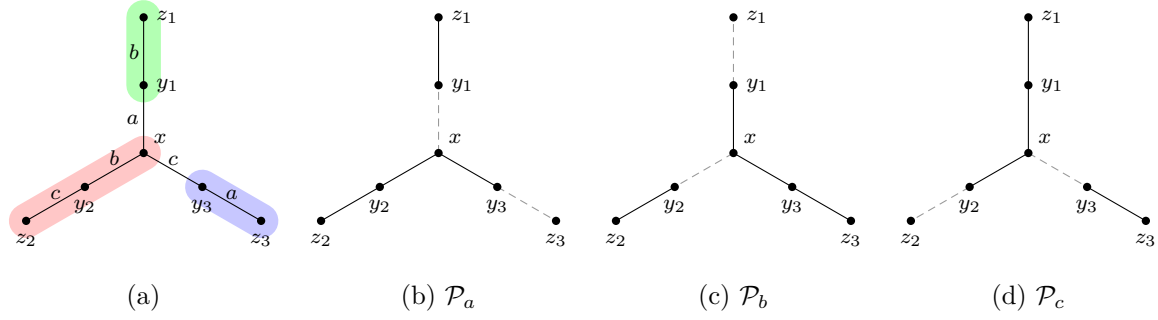
\begin{figure}[t]
\centering
\begin{tikzpicture}[scale=0.78,
  vtx/.style={circle,fill=black,inner sep=1.1pt},
  lab/.style={inner sep=0.6pt,font=\scriptsize},
  gone/.style={densely dashed,gray!80}]
\newcommand{\spidercoords}{%
  \coordinate (x) at (0,0);
  \coordinate (y1) at (90:1.15);  \coordinate (z1) at (90:2.3);
  \coordinate (y2) at (210:1.15); \coordinate (z2) at (210:2.3);
  \coordinate (y3) at (330:1.15); \coordinate (z3) at (330:2.3);}
\newcommand{\spidervertices}{%
  \foreach \v in {x,y1,z1,y2,z2,y3,z3} \node[vtx] at (\v) {};
  \node[font=\scriptsize,above right=0pt] at (x) {$x$};
  \node[font=\scriptsize,right=1pt] at (y1) {$y_1$}; \node[font=\scriptsize,right=1pt] at (z1) {$z_1$};
  \node[font=\scriptsize,below=1pt] at (y2) {$y_2$}; \node[font=\scriptsize,below=1pt] at (z2) {$z_2$};
  \node[font=\scriptsize,below=1pt] at (y3) {$y_3$}; \node[font=\scriptsize,below=1pt] at (z3) {$z_3$};}
\begin{scope}
\spidercoords
\draw[line width=13pt,line cap=round,red!22] (x)--(z2);
\draw[line width=13pt,line cap=round,blue!22] (y3)--(z3);
\draw[line width=13pt,line cap=round,green!30] (y1)--(z1);
\draw (x)--(y1) node[lab,midway,left=1pt]{$a$};
\draw (y1)--(z1) node[lab,midway,left=1pt]{$b$};
\draw (x)--(y2) node[lab,midway,above=1pt]{$b$};
\draw (y2)--(z2) node[lab,midway,above=1pt]{$c$};
\draw (x)--(y3) node[lab,midway,above=1pt]{$c$};
\draw (y3)--(z3) node[lab,midway,above=1pt]{$a$};
\spidervertices
\node at (0,-2.45) {\small (a)};
\end{scope}
\begin{scope}[xshift=5.0cm]
\spidercoords
\draw[gone] (x)--(y1); \draw (y1)--(z1); \draw (x)--(y2); \draw (y2)--(z2); \draw (x)--(y3); \draw[gone] (y3)--(z3);
\spidervertices
\node at (0,-2.45) {\small (b) $\mathcal P_a$};
\end{scope}
\begin{scope}[xshift=10.0cm]
\spidercoords
\draw (x)--(y1); \draw[gone] (y1)--(z1); \draw[gone] (x)--(y2); \draw (y2)--(z2); \draw (x)--(y3); \draw (y3)--(z3);
\spidervertices
\node at (0,-2.45) {\small (c) $\mathcal P_b$};
\end{scope}
\begin{scope}[xshift=15.0cm]
\spidercoords
\draw (x)--(y1); \draw (y1)--(z1); \draw (x)--(y2); \draw[gone] (y2)--(z2); \draw[gone] (x)--(y3); \draw (y3)--(z3);
\spidervertices
\node at (0,-2.45) {\small (d) $\mathcal P_c$};
\end{scope}
\end{tikzpicture}
\caption{(a) A tree $T$ whose edges carry colours $a,b,c$, each colour on two edges, together with the rainbow partition $B_a=\{x,y_2,z_2\}$, $B_b=\{y_3,z_3\}$, $B_c=\{y_1,z_1\}$ (shaded). (b)--(d) The connected partitions $\mathcal P_a,\mathcal P_b,\mathcal P_c$ formed by the components of $T-E_a$, $T-E_b$, $T-E_c$, where $E_\gamma$ denotes the set of edges of colour $\gamma$; the deleted edges are dashed. For every colour $\gamma\in\{a,b,c\}$, the set $B_\gamma$ is contained in a part of $\mathcal P_\gamma$, and $T[B_\gamma]$ contains no edge of colour $\gamma$.}
\label{fig:spider}
\end{figure}

A connected partition of a tree into $q$ parts is determined by the $q-1$ edges joining different parts, and conversely, deleting a set of edges from a tree yields the connected partition of its vertex set into the components of the resulting forest. This gives an equivalent formulation in terms of edge colourings in which an edge may carry several colours (Proposition~\ref{prop:equivalence}).

\begin{theorem}[Edge-colour form]
\label{thm:Aprime}
Let $T$ be a finite tree, let $k\ge1$, and let $E_1,\ldots,E_k\subseteq E(T)$, not necessarily disjoint, satisfy $|E_i|\le k-1$ for every $i\in[k]$. Then there is an allocation $(B_i)_{i\in[k]}$ of $V(T)$ such that every nonempty $B_i$ is connected and the induced subgraph $T[B_i]$ contains no edge of $E_i$. In particular, one can choose a component $K_i$ of $T-E_i$ for every $i\in[k]$ so that $K_1\cup\cdots\cup K_k=V(T)$.
\end{theorem}

Thus, if each of $k$ colours occurs on at most $k-1$ edges of a tree, the vertices can be partitioned into connected pieces labelled by distinct colours so that no piece contains an edge of its own colour.

For the fair-division consequence, let the vertices of a tree $T$ be chores, and let $n$ agents have \emph{monotone costs} $\kappa_i\colon2^{V(T)}\to\R_{\ge0}$, that is, $\kappa_i(\varnothing)=0$ and $\kappa_i(X)\le\kappa_i(Y)$ whenever $X\subseteq Y$. Additive costs $\kappa_i(X)=\sum_{v\in X}\kappa_i(v)$, where $\kappa_i(v)=\kappa_i(\{v\})$, are the main example. Let $\Pi_n(T)$ be the set of connected partitions of $V(T)$ with at most $n$ parts, so that the minimum in~\eqref{eq:intro-mms} is taken over $\Pi_n(T)$. A \emph{connected allocation} is an allocation $(B_i)_{i\in[n]}$ of $V(T)$ in which every nonempty bundle $B_i$ is connected; it is an \emph{MMS allocation} if $\kappa_i(B_i)\le\MMS_i(T,n)$ for every $i\in[n]$.

\begin{theorem}[Connected MMS allocations of chores on trees]
\label{thm:B}
Let $T$ be a finite tree and let $n\ge1$.
\begin{enumerate}[label=\textup{(\alph*)}]
\item For all monotone costs $\kappa_1,\ldots,\kappa_n$ there is a connected MMS allocation. More strongly, if $\mathcal P_i\in\Pi_n(T)$ attains $\MMS_i(T,n)$ for every $i$, then there is a connected allocation in which every nonempty $B_i$ is contained in a member of $\mathcal P_i$. If $|V(T)|\ge n$, all bundles can be chosen nonempty.
\item Conversely, if connected MMS allocations exist for all additive costs $\kappa_1,\ldots,\kappa_n$ on $V(T)$, then the conclusion of Theorem~\ref{thm:A} holds for $T$ and $k=n$.
\end{enumerate}
\end{theorem}

Part~(b) shows that, for additive costs, the fair-division statement on a given tree is equivalent to Theorem~\ref{thm:A} on that tree; Theorem~\ref{thm:A} is therefore exactly the combinatorial content of the question of Xiao, Qiu, and Huang. The benchmark~\eqref{eq:intro-mms} coincides with theirs whenever $|V(T)|\ge n$ (Remark~\ref{rem:xqh}). Since it is defined by connected partitions, it is at least the unconstrained maximin share of chores, and Theorem~\ref{thm:B} says nothing about the unconstrained problem, in which exact MMS allocations of chores need not exist~\cite{AzizEtAl2017}.

The main tool in the proof of Theorem~\ref{thm:Aprime} is a statement about orientations. An \emph{orientation} $\omega$ of a tree $Q$ chooses one of the two directions of every edge; if $\omega$ directs an edge from $u$ to $v$, then $u$ is its \emph{tail}, $v$ is its \emph{head}, and the edge points towards $v$ and away from $u$. For $Z\subseteq V(Q)$, let $\bd Z$ be the set of edges with exactly one endpoint in $Z$. A connected set $Z$ is \emph{$\omega$-inward} if $\omega$ directs every edge of $\bd Z$ towards its endpoint in $Z$.

\begin{theorem}[Inward partitions of oriented trees]
\label{thm:C}
Let $Q$ be a tree with $s\ge2$ vertices, and let $\omega_1,\ldots,\omega_{s-1}$ be orientations of $Q$, not necessarily distinct. Then there are a connected partition $\{Z_1,\ldots,Z_r\}$ of $V(Q)$ with $r\ge2$ and a partition $\{0,1,\ldots,s-1\}=G_1\cup\cdots\cup G_r$ into pairwise disjoint sets such that $|G_j|=|Z_j|$ for every $j\in[r]$, and $Z_j$ is $\omega_c$-inward for every $j\in[r]$ and every $c\in G_j\setminus\{0\}$.
\end{theorem}

The trivial partition $\{V(Q)\}$ satisfies the inward condition vacuously, so the content of Theorem~\ref{thm:C} lies in the requirement $r\ge2$. The index $0$ carries no requirement, and it cannot be replaced by an $s$-th orientation (Remark~\ref{rem:C-sharp}). Theorem~\ref{thm:C} yields a balanced splitting of trees (Corollary~\ref{cor:split}): if $\mu_1,\ldots,\mu_{s-1}$ are real functions on the vertices of $Q$, each with total sum zero, then the partitions can be chosen so that $\sum_{v\in Z_j}\mu_c(v)\ge0$ for every $j$ and every $c\in G_j\setminus\{0\}$.

We also show that Theorem~\ref{thm:A} is sharp. Already on a path, one partition with $k+1$ parts and $k-1$ partitions with $k$ parts may admit no choice of one part from each partition covering all vertices (Proposition~\ref{prop:sharp}). For paths, Theorem~\ref{thm:Aprime} is equivalent to the elementary fact that a word in which each of $k$ letters occurs at most $k-1$ times misses some ordering of the letters as a subsequence (Proposition~\ref{prop:paths}). Known short words containing all permutations show that the individual bounds $|E_i|\le k-1$ cannot be replaced by a bound on their sum (Proposition~\ref{prop:average}). On cycles, partitions into at most $k-1$ parts suffice, and this is best possible for every $k\ge2$ (Proposition~\ref{prop:cycles}). In Section~\ref{sec:chores} we prove the packing counterpart for goods (Proposition~\ref{prop:goods}). Moreover, a connected graph satisfies the conclusion of Theorem~\ref{thm:A} for all $k$, or its packing counterpart, if and only if it is a triangular cactus, that is, if it contains no cycle of length at least four (Theorem~\ref{thm:cacti}). Finally, a rainbow partition can be computed in time $k^{O(k)}$ plus polynomial time, and a connected MMS allocation of additive chores on a tree in time $n^{O(n)}$ plus polynomial time (Theorem~\ref{thm:algorithm} and Corollary~\ref{cor:mms-algorithm}).

\subsection{Why trees, and why chores are harder than goods}
\label{sec:why}

\parhead{Packing versus covering.}
Suppose that every agent fixes a connected partition $\mathcal P_i$ attaining its maximin share, with at most $n$ parts for chores and exactly $n$ parts for goods (Section~\ref{sec:chores}). For goods with monotone valuations, every connected bundle containing a part of $\mathcal P_i$ is acceptable to agent $i$; for chores with monotone costs, every connected subset of a part is acceptable. Goods, therefore, lead to a \emph{packing} problem, in which pairwise disjoint connected supersets of parts must be chosen, one for each agent, whereas chores lead to a \emph{covering} problem, in which connected subsets of parts must exhaust all vertices. On trees, the packing problem has a short greedy solution: a deepest vertex whose subtree contains a part of some partition can be cut off together with its subtree, and every other partition loses at most one part (Proposition~\ref{prop:goods}, a combinatorial form of the argument in~\cite{BouveretEtAl2017}). The covering problem does not reduce in this way. Cutting off an acceptable bundle may disconnect the remaining tree~\cite[Section~3]{XiaoEtAl2023arXiv}, and the restrictions of the other partitions to the resulting components may then have too many parts for the remaining agents. Theorem~\ref{thm:A} shows that, nevertheless, the numerical condition $|\mathcal P_i|\le k$ suffices.

\parhead{The host tree.}
Let the vertices $1,2,4,3$ form a four-cycle in this cyclic order, and let
\[
 \mathcal P_1=\bigl\{\{1,2\},\{3,4\}\bigr\},\qquad
 \mathcal P_2=\bigl\{\{1,3\},\{2,4\}\bigr\};
\]
see Figure~\ref{fig:small}(a). Both are partitions of the cycle into two connected parts, and every part of $\mathcal P_1$ meets every part of $\mathcal P_2$ in exactly one vertex. Hence, no choice of one part from each partition covers the four vertices, and no two disjoint parts can be chosen either. The example defeats both the covering and the packing versions. For cycles, connected MMS allocations of goods need not exist already on an eight-cycle~\cite{BouveretEtAl2017}. For chores on cycles, Xiao, Qiu, and Huang~\cite[Section~7]{XiaoEtAl2023arXiv} studied approximate connected MMS allocations. For the combinatorial statement, Proposition~\ref{prop:cycles} shows that on cycles exactly one part fewer is needed: $k-1$ parts suffice, and $k$ parts do not suffice for any $k\ge2$.

\parhead{Paths.}
On a path, the components of $T-E_i$ are intervals, and Theorem~\ref{thm:Aprime} becomes a statement about words: the colours of the edges, read from left to right, form a word in which a greedy allocation succeeds exactly when some ordering of the colours is not a subsequence (Proposition~\ref{prop:paths} and Corollary~\ref{cor:paths}). A word in which every letter occurs at most $k-1$ times cannot contain all $k!$ orderings; this follows from a variant of the elementary greedy lower-bound argument of Kleitman and Kwiatkowski~\cite{KleitmanKwiatkowski1976} (see~\cite[Section~4]{EngenVatter2021}), and we include the short proof in Corollary~\ref{cor:paths}. For paths, the theorem therefore has a short proof. The difficulty lies in the branching of trees.

\parhead{Covers versus partitions.}
A choice of components covering the vertices does not, in general, yield a partition of the vertices into connected subsets of the chosen components. In the star of Figure~\ref{fig:small}(b), the components $K_1=\{a,x,b\}$ of $T-E_1$ and $K_2=\{c,x,d\}$ of $T-E_2$ cover all vertices, but there are no disjoint connected sets $B_1\subseteq K_1$ and $B_2\subseteq K_2$ covering the vertices: the centre $x$ must belong to one of them, and the other then contains two leaves but not the centre. The partition statement is needed for fair division because bundles must be disjoint (the edge sets in this example have two edges each, so it lies outside the range of Theorem~\ref{thm:Aprime}; Section~\ref{sec:covers} gives an example within this range).

\begin{figure}[t]
\centering
\begin{tikzpicture}[scale=1.25,vtx/.style={circle,fill=black,inner sep=1.3pt}]
\begin{scope}
\coordinate (1) at (0,1); \coordinate (2) at (1,1); \coordinate (4) at (1,0); \coordinate (3) at (0,0);
\draw[line width=12pt,line cap=round,red!35,opacity=0.6] (1)--(2);
\draw[line width=12pt,line cap=round,red!35,opacity=0.6] (3)--(4);
\draw[line width=12pt,line cap=round,blue!35,opacity=0.6] (1)--(3);
\draw[line width=12pt,line cap=round,blue!35,opacity=0.6] (2)--(4);
\draw (1)--(2)--(4)--(3)--cycle;
\foreach \v in {1,2,3,4} \node[vtx] at (\v) {};
\node[above left,font=\scriptsize] at (1) {$1$}; \node[above right,font=\scriptsize] at (2) {$2$};
\node[below right,font=\scriptsize] at (4) {$4$}; \node[below left,font=\scriptsize] at (3) {$3$};
\node at (0.5,-0.55) {\small (a)};
\end{scope}
\begin{scope}[xshift=4.2cm,yshift=0.5cm]
\coordinate (x) at (0,0);
\coordinate (a) at (-1.1,0.55); \coordinate (b) at (-1.1,-0.55);
\coordinate (c) at (1.1,0.55); \coordinate (d) at (1.1,-0.55);
\draw[line width=12pt,line cap=round,line join=round,red!35,opacity=0.6] (a)--(x)--(b);
\draw[line width=12pt,line cap=round,line join=round,blue!35,opacity=0.6] (c)--(x)--(d);
\draw (x)--(a) node[midway,above,font=\scriptsize]{$2$};
\draw (x)--(b) node[midway,below,font=\scriptsize]{$2$};
\draw (x)--(c) node[midway,above,font=\scriptsize]{$1$};
\draw (x)--(d) node[midway,below,font=\scriptsize]{$1$};
\foreach \v in {x,a,b,c,d} \node[vtx] at (\v) {};
\node[left,font=\scriptsize] at (a) {$a$}; \node[left,font=\scriptsize] at (b) {$b$}; \node[right,font=\scriptsize] at (c) {$c$}; \node[right,font=\scriptsize] at (d) {$d$};
\node[above=2pt,font=\scriptsize] at (x) {$x$};
\node at (0,-1.05) {\small (b)};
\end{scope}
\end{tikzpicture}
\caption{(a) Two partitions of the four-cycle $1\,2\,4\,3$ into two connected parts: $\mathcal P_1$ (horizontal) and $\mathcal P_2$ (vertical). No part of $\mathcal P_1$ together with a part of $\mathcal P_2$ covers the cycle. (b) A star with $E_1=\{xc,xd\}$ and $E_2=\{xa,xb\}$ (edge labels). The components $K_1=\{a,x,b\}$ of $T-E_1$ and $K_2=\{c,x,d\}$ of $T-E_2$ (shaded) cover the star, but the star has no partition into connected sets $B_1\subseteq K_1$ and $B_2\subseteq K_2$.}
\label{fig:small}
\end{figure}
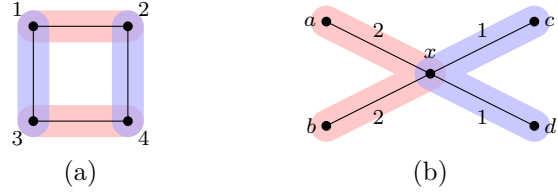

\parhead{Topology.}
Topological methods have been very successful for connected fair division on paths, from envy-free cake cutting to its discrete versions (Section~\ref{sec:related}). For cakes shaped like trees, however, the continuous statements fail: for a cake shaped like the letter Y, an envy-free connected division need not exist even for two agents, as pointed out by Stromquist (see~\cite{IgarashiZwicker2024}); for KKM- and Sperner-type theorems on trees see~\cite{NiedermaierEtAl2014}. We therefore do not apply topology to divisions of the tree itself. Instead, it is applied to a simplex of weights on the edges of an auxiliary tree.

\subsection{Proof strategy}
\label{sec:strategy}

After standard normalisations (contracting uncoloured edges, subdividing edges with several colours, and adding pendant edges), we may assume in Theorem~\ref{thm:Aprime} that every edge has exactly one colour and every colour occurs on exactly $k-1$ edges. We then distinguish one colour, say $0$, and contract the components of the forest $T-E_0$. The quotient $Q$ is a tree on $k$ vertices whose edges are the edges of colour $0$. Every other colour $c$ has all its $k-1$ edges inside the contracted components, and we record their numbers $e_c(H)$ at the vertices $H$ of $Q$. It suffices to split $Q$ into $r\ge2$ connected parts $Z_j$ and the colours into groups $G_j$ with $|G_j|=|Z_j|$, so that every colour $c\in G_j\setminus\{0\}$ has at most $|Z_j|-1$ edges inside the preimage of $Z_j$. Colour $0$ automatically has exactly $|Z_j|-1$ edges there, and induction on the number of colours applies to every part; see Figure~\ref{fig:reduction}.

To find such a splitting, orient every edge of $Q$, for each colour $c\ne0$, towards the side on which colour $c$ has slack, that is, fewer edges than that side has vertices; exactly one side has slack, since the two sides together have $k$ vertices and carry $k-1$ edges of colour $c$. If a connected set is inward for this orientation, then the budget condition for $c$ holds on it (see the proofs of Corollaries~\ref{cor:split} and~\ref{cor:integer}). This is how Theorem~\ref{thm:C} enters.

To prove Theorem~\ref{thm:C}, we attach nonnegative weights $w$ to the edges of $Q$. For an orientation $\omega$ and a vertex $y$, let $\phi_w(\omega,y)$ be the total weight of the edges pointing away from $y$, that is, the weight that must be reversed to direct every edge towards $y$. The minimizers of $\phi_w(\omega,\cdot)$ form a union of components of $Q-\supp(w)$, and each such component is $\omega$-inward (Lemma~\ref{lem:min}). Hence, it suffices to find common weights $w^*\ne0$ at which the $s-1$ orientations have minimisers at pairwise distinct vertices (Theorem~\ref{thm:weights}); the components of $Q-\supp(w^*)$ are then the parts required in Theorem~\ref{thm:C}.

We find $w^*$ by means of Gale's colourful KKM theorem on the simplex of edge weights. Every orientation labels every weight vector $w$ by an edge of $\supp(w)$ on the boundary of a component of $Q-\supp(w)$ consisting of minimisers of $\phi_w(\omega,\cdot)$, and the theorem provides a point $w^*$ near which distinct orientations realise distinct labels. The difficulty is that weights may vanish at the limit point $w^*$. Components then merge, and a merged component $Z$ must accommodate every orientation whose label is an edge inside $Z$ or on its boundary. If every orientation could choose its boundary label freely, there could be $|Z|-1+|\bd Z|$ such orientations, for instance $s-1$ orientations whose minimisers lie at the centre of a star, while only $|Z|$ vertices are available. We resolve this with a labelling that does not depend on the orientation and is determined by a leaf-elimination order of the edges: a component is labelled by the first edge on its way towards the largest edge of the current support. The key property (Lemma~\ref{lem:compatibility}) is that, for every coarsening, the labels of all finer components inside a merged component $Z$ lie in a fixed set of exactly $|Z|$ edges, namely the internal edges of $Z$ and one designated boundary edge (Figure~\ref{fig:labels}). This local capacity count converts the colourful point into the assignment required in Theorem~\ref{thm:weights}.

\subsection{Related work}
\label{sec:related}

\parhead{Connected fair division.}
The graph model of Bouveret et~al.~\cite{BouveretEtAl2017} has been studied from many angles. For goods on cycles, Lonc and Truszczyński~\cite{LoncTruszczynski2020} studied existence and approximation of MMS allocations; Greco and Scarcello~\cite{GrecoScarcello2020} studied the complexity of computing connected MMS allocations on general graphs, and Bla\v{z}ej et~al.~\cite{BlazejEtAl2025} constant-factor approximations on several graph classes. Bei, Igarashi, Lu, and Suksompong~\cite{BeiEtAl2022} compared connected and unconstrained maximin shares through the price of connectivity. Envy-freeness up to one item with connected bundles was studied by Bil\`o et~al.~\cite{BiloEtAl2022}, and on paths it was established for any number of agents by Igarashi~\cite{Igarashi2023}. Pareto optimality with connected bundles was studied by Igarashi and Peters~\cite{IgarashiPeters2019}, and the parameterised complexity of connected fair division by Deligkas et~al.~\cite{DeligkasEtAl2021}. For items on a line, contiguous allocations of goods and of chores were studied in~\cite{Suksompong2019} and~\cite{HohneVanStee2021}. For divisible resources placed on graphs, Bei et~al.~\cite{BeiEtAl2025} studied proportional divisions of graphical cakes, and Elkind, Segal-Halevi, and Suksompong~\cite{ElkindEtAl2021} proved that MMS allocations of graphical cakes exist when the graph is a forest. Igarashi and Zwicker~\cite{IgarashiZwicker2024} studied envy-free divisions of cakes formed by intervals glued together (tangles), which also shows how branching affects topological existence results.

\parhead{Chores.}
Bouveret, Cechl\'arov\'a, and Lesca~\cite{BouveretEtAl2019} initiated the study of chores on graphs, for proportionality, envy-freeness, and equitability, and showed that chore division differs from the division of goods also in this setting. Xiao, Qiu, and Huang~\cite{XiaoEtAl2023,XiaoEtAl2023arXiv} studied connected MMS chore allocations. Besides the results on trees mentioned above, they obtained connected $7/6$-approximate MMS allocations of chores for three agents on a cycle by a linear-programming method, and they noted that, for additive costs, the benchmarks can be computed in polynomial time on trees and on cycles~\cite[Lemma~2.1]{XiaoEtAl2023arXiv}. Sun and Li~\cite{SunLi2026} proved that connected allocations of chores on a path that are envy-free up to one item always exist. Without connectivity constraints, MMS allocations of chores need not exist~\cite{AzizEtAl2017}; approximation guarantees were improved in a sequence of papers, including~\cite{AzizEtAl2017,GargEtAl2025,HuangLu2021,HuangSegalHalevi2023}, and lower bounds were obtained by Feige, Sapir, and Tauber~\cite{FeigeEtAl2021}. For goods without connectivity constraints, MMS allocations also need not exist~\cite{KurokawaEtAl2018}. We refer to the surveys~\cite{AmanatidisEtAl2023,GuoEtAl2023}. Our positive result rests on the compatibility of all benchmark partitions with one common tree. It does not involve any comparison of the numerical costs of bundles.

\parhead{Topological methods.}
Sperner's lemma~\cite{Sperner1928} and the KKM theorem~\cite{KKM1929} underlie many existence results in fair division. Stromquist~\cite{Stromquist1980} and Woodall~\cite{Woodall1980} proved that envy-free divisions of an interval into connected pieces exist. Su~\cite{Su1999} popularised the Sperner-lemma approach to rental harmony and cake cutting, with vertices of a triangulation owned by different agents, an idea that he credits to Simmons. The colourful KKM theorem of Gale~\cite{Gale1984} and the permutation-based Sperner lemma of Bapat~\cite{Bapat1989} formalise the idea that distinct agents can receive distinct labels; see Asada et~al.~\cite{AsadaEtAl2018} for generalisations and further applications to fair division. Meunier and Su~\cite{MeunierSu2019} proved further multilabelled versions of Sperner's lemma and of Fan's lemma, with applications to fair division. Asada et~al.~\cite{AsadaEtAl2018} showed that an envy-free rent division can be found from the preferences of only $n-1$ of the $n$ roommates, and Frick, Houston-Edwards, and Meunier~\cite{FrickEtAl2019} gave an elementary proof of this secretive-roommate theorem. Discrete connected divisions on paths were obtained by related arguments in~\cite{BiloEtAl2022,Igarashi2023,SunLi2026}. In our proof, the simplex parametrizes weights on the edges of an auxiliary tree rather than divisions of the items, and the labels are discontinuous functions of the weights. The tree-specific ingredient is a combinatorial compatibility property of the labels under coarsening (Lemma~\ref{lem:compatibility}).

\parhead{Cooperative colourings.}
A cooperative colouring of graphs $\Gamma_1,\ldots,\Gamma_k$ on a common vertex set is a choice of an independent set in each $\Gamma_i$ such that these sets cover all vertices~\cite{AharoniEtAl2015}. Given connected partitions $\mathcal P_i$, let $\Gamma_i$ be the complete multipartite graph whose parts are the members of $\mathcal P_i$. Its independent sets are the subsets of single members of $\mathcal P_i$, and its chromatic number is $|\mathcal P_i|$. Theorem~\ref{thm:A} thus asserts that graphs $\Gamma_1,\ldots,\Gamma_k$ on $V(T)$ admit a cooperative colouring, even one by disjoint sets whose nonempty members are connected in $T$, whenever each $\Gamma_i$ has a proper $k$-colouring whose colour classes are connected in a common tree $T$. Most known sufficient conditions for cooperative colourings bound the maximum degree~\cite{AharoniEtAl2015,AharoniEtAl2020,Bradshaw2023}, in the spirit of Haxell's theorem on independent transversals~\cite{Haxell2001}; the graphs $\Gamma_i$ above are dense. Bartnicki, Czerwi\'nski, Grytczuk, and Miechowicz~\cite{BartnickiEtAl2023} proved that $k$ matroids on a common ground set admit a cooperative colouring whenever each of them is $k$-colourable. For graphs, $k$-colourability alone does not suffice: the partitions of the four-cycle in Figure~\ref{fig:small}(a) correspond to two copies of $K_{2,2}$. Theorem~\ref{thm:A} identifies a natural non-matroidal class for which $k$-colourability does suffice.

\parhead{Monochromatic partitions.}
Colour every edge $e$ of $T$ by the set of those $i\in[k]$ with $e\notin E_i$; this colouring is complementary to the one in Theorem~\ref{thm:Aprime}. Then the edges of colour $i$ form the spanning forest $T-E_i$, which has at most $k$ components, and Theorem~\ref{thm:Aprime} states that $V(T)$ can be partitioned into at most $k$ monochromatic trees of pairwise distinct colours; single vertices count as trees of every colour. Partitions and covers of edge-coloured graphs by monochromatic trees are a classical topic~\cite{ErdosEtAl1991,Gyarfas2016}. Haxell and Kohayakawa~\cite{HaxellKohayakawa1996} showed that for large $N$ every $r$-edge-colouring of $K_N$ admits a partition into at most $r$ monochromatic trees of distinct colours. Bal and DeBiasio~\cite{BalDeBiasio2017} studied this partition property and its covering analogue for other host graphs, and Buci\'c, Kor\'andi, and Sudakov~\cite{BucicEtAl2021} related the covering version to transversal covers of auxiliary hypergraphs. Those results concern complete, random, or dense host graphs. In our setting, the host graph is a tree, and the hypothesis bounds the number of components of each colour class.

\parhead{Universal words and tree partitioning.}
The path case of our results is connected with words containing all permutations of an alphabet as subsequences~\cite{Newey1973,Adleman1974,KoutasHu1975,KleitmanKwiatkowski1976,Radomirovic2012}; see the survey of Engen and Vatter~\cite{EngenVatter2021}. Finally, for additive costs and $|V(T)|\ge n$, computing a benchmark $\MMS_i(T,n)$ is the classical min--max tree partitioning problem (Lemma~\ref{lem:exactly}), solved in polynomial time by Becker, Schach, and Perl~\cite{BeckerEtAl1982} and in linear time by Frederickson~\cite{Frederickson1991}; see also~\cite{KunduMisra1977,PerlSchach1981} for related tree partitioning problems.

\subsection{Organization}
Section~\ref{sec:prelim} collects some facts about trees, the equivalence of Theorem~\ref{thm:A} with the partition statement of Theorem~\ref{thm:Aprime}, the normalisation, and the passage to nonempty pieces. Section~\ref{sec:reduction} derives Theorems~\ref{thm:A} and~\ref{thm:Aprime} from Theorem~\ref{thm:C}. Section~\ref{sec:orientation} proves Theorem~\ref{thm:C}. Section~\ref{sec:sharpness} discusses sharpness, paths, and cycles, Section~\ref{sec:chores} the fair-division consequences, Section~\ref{sec:beyond} graphs other than trees, and Section~\ref{sec:algorithm} algorithms. Section~\ref{sec:open} lists open problems.

\section{Preliminaries}
\label{sec:prelim}

\subsection{Tree facts}
For a graph $G$, a set $X\subseteq V(G)$, and a set $F\subseteq E(G)$, we write $G[X]$ for the induced subgraph, $G-X=G[V(G)\setminus X]$, and $G-F$ for the graph obtained by deleting the edges of $F$ but keeping all vertices. A component is identified with its vertex set, and $\Comp(G)$ denotes the set of components of $G$. A tree is a nonempty connected acyclic graph.

Let $Q=(W,D)$ be a tree. For $F\subseteq D$ we write
\[
 \CC F=\Comp(Q-F),
 \qquad Z_F(y)=\text{the member of $\CC F$ containing $y$}\quad(y\in W),
\]
and call the members of $\CC F$ the \emph{$F$-parts}. For $Z\subseteq W$, let $\inn Z$ be the set of edges with both endpoints in $Z$, and $\bd Z$ the set of edges with exactly one endpoint in $Z$. For a connected partition $\mathcal P$ of $W$, the \emph{quotient} $Q/\mathcal P$ is the graph with vertex set $\mathcal P$ in which distinct $P,P'\in\mathcal P$ are adjacent if some edge of $Q$ joins $P$ and $P'$, and the \emph{quotient map} $\pi_{\mathcal P}\colon W\to\mathcal P$ sends a vertex to the part containing it.

\begin{lemma}
\label{lem:tree}
Let $Q=(W,D)$ be a tree.
\begin{enumerate}
\item A connected set contains the path of $Q$ between any two of its vertices. The intersection of two connected sets is connected or empty.
\item Let $Z\subseteq W$ be connected and let $H\in\Comp(Q-Z)$. Exactly one edge $f_H=qh$ of $Q$, with $q\in Z$ and $h\in H$, joins $Z$ and $H$, and $H$ is the component of $Q-f_H$ not containing $q$. The map $H\mapsto f_H$ is a bijection from $\Comp(Q-Z)$ onto $\bd Z$. For $z\in Z$ and $t\in H$, the path from $z$ to $t$ consists of the path from $z$ to $q$, which lies in $Z$, the edge $f_H$, and the path from $h$ to $t$, which lies in $H$. In particular, $q$ is the unique vertex of $Z$ closest to $t$, and $f_H$ is the first edge of the path from $q$ to $t$.
\item If $\mathcal P$ is a connected partition of $W$, then $Q/\mathcal P$ is a tree, exactly $|\mathcal P|-1$ edges of $Q$ join distinct parts and they correspond bijectively to the edges of $Q/\mathcal P$, and the preimage under $\pi_{\mathcal P}$ of a connected set of vertices of $Q/\mathcal P$ is connected in $Q$.
\item For $F\subseteq D$, the set $\CC F$ has exactly $|F|+1$ members, and each of them is connected in $Q$. Every edge of $F$ joins two distinct $F$-parts, and every edge outside $F$ has both endpoints in one $F$-part. Consequently $\bd Z\subseteq F$ and $\inn Z\cap F=\varnothing$ for every $Z\in\CC F$, and $Q/\CC F$ is a tree whose edges correspond bijectively to $F$. If $F\subseteq J\subseteq D$, then every $J$-part is contained in an $F$-part.
\end{enumerate}
\end{lemma}

\begin{proof}
(i) If $X$ is connected and $x,y\in X$, then $Q[X]$ contains an $x$--$y$ path, which is the unique $x$--$y$ path of $Q$. Conversely, a nonempty set containing the path between any two of its vertices is connected. The intersection of two sets with this property has it as well.

(ii) Since $Q$ is connected and $H\ne W$, some edge has exactly one endpoint in $H$. Its other endpoint lies in no other component of $Q-Z$, hence in $Z$. Suppose that $h_1z_1\ne h_2z_2$ are two edges with $h_1,h_2\in H$ and $z_1,z_2\in Z$. The path from $h_2$ to $h_1$ in $Q[H]$, the edge $h_1z_1$, the path from $z_1$ to $z_2$ in $Q[Z]$, and the edge $z_2h_2$ form a cycle, which is impossible. Hence, $f_H$ is the only edge with exactly one endpoint in $H$. As $H$ is connected in $Q-f_H$ and no edge of $Q-f_H$ leaves $H$, it is a component of $Q-f_H$, and it does not contain $q$. Every edge of $\bd Z$ has its outer endpoint in exactly one component of $Q-Z$, which gives the bijection. Finally, the walk described in the statement has pairwise distinct vertices because $Z\cap H=\varnothing$, so it is the path from $z$ to $t$. Its length is $\operatorname{dist}(z,q)+1+\operatorname{dist}(h,t)$, which is minimal exactly for $z=q$.

(iii) The quotient is connected because $Q$ is. Two distinct edges of $Q$ joining the same two parts $P\ne P'$ would form a cycle together with paths inside $P$ and $P'$. If $\mathcal P$ has $r$ parts, then each $Q[P]$ is a tree, so exactly $\sum_{P\in\mathcal P}(|P|-1)=|W|-r$ edges of $Q$ lie inside parts, and the remaining $r-1$ edges give $r-1$ distinct edges of the quotient. A connected graph with $r$ vertices and $r-1$ edges is a tree. If $S$ is a connected set of vertices of $Q/\mathcal P$ and $u,v\in\pi_{\mathcal P}^{-1}(S)$, then a path from $\pi_{\mathcal P}(u)$ to $\pi_{\mathcal P}(v)$ in $(Q/\mathcal P)[S]$ lifts to a walk from $u$ to $v$ in $Q[\pi_{\mathcal P}^{-1}(S)]$: replace every quotient edge by the corresponding edge of $Q$ and join consecutive endpoints by paths inside the parts.

(iv) Deleting an edge of a forest increases the number of components by exactly one: its endpoints are joined by no other path, and every other vertex remains joined to one of them. Hence $|\CC F|=|F|+1$. Each $F$-part is connected in $Q-F$, hence in $Q$. An edge of $F$ with both endpoints in one $F$-part would form a cycle together with a path in that part, and an edge outside $F$ is an edge of $Q-F$. This gives $\bd Z\subseteq F$ and $\inn Z\cap F=\varnothing$, and the statement about $Q/\CC F$ follows from~(iii). If $F\subseteq J$, then $Q-J$ is a subgraph of $Q-F$, so every $J$-part is connected in $Q-F$ and lies in one $F$-part.
\end{proof}

\subsection{The edge-colour form}
In the proofs it is convenient to index colours by an arbitrary finite set $C$ with $|C|=k$. An \emph{instance} $(T,(E_c)_{c\in C})$ consists of a tree $T$ and edge sets $E_c\subseteq E(T)$; it is \emph{admissible} if $|E_c|\le k-1$ for all $c\in C$. A \emph{solution} is an allocation $(B_c)_{c\in C}$ of $V(T)$ such that every nonempty $B_c$ is connected and
\begin{equation}
 E(T[B_c])\cap E_c=\varnothing\qquad(c\in C).
 \label{eq:avoid}
\end{equation}
The partition statement of Theorem~\ref{thm:Aprime} asserts that every admissible instance has a solution.

\begin{proposition}
\label{prop:equivalence}
For every tree $T$ and every $k\ge1$, the first statement of Theorem~\ref{thm:A} for $T$ and $k$ is equivalent to the partition statement of Theorem~\ref{thm:Aprime} for $T$ and $k$. The partition statement of Theorem~\ref{thm:Aprime} implies its covering statement.
\end{proposition}

\begin{proof}
Given edge sets $E_1,\ldots,E_k$ with $|E_i|\le k-1$, the partitions $\mathcal P_i=\Comp(T-E_i)$ are connected and have at most $k$ members by Lemma~\ref{lem:tree}(iv). A connected subset of a member of $\mathcal P_i$ contains no edge of $E_i$, since $E_i$ contains no edge with both endpoints in one member. Hence a rainbow partition subordinate to $(\mathcal P_i)$ is a solution.

Conversely, given connected partitions $\mathcal P_i$ with $|\mathcal P_i|\le k$, let $E_i$ be the set of edges of $T$ joining distinct members of $\mathcal P_i$. By Lemma~\ref{lem:tree}(iii), $|E_i|=|\mathcal P_i|-1\le k-1$. Every member of $\mathcal P_i$ is connected in $T-E_i$, because its internal edges avoid $E_i$, and no edge of $T-E_i$ leaves it; thus $\Comp(T-E_i)=\mathcal P_i$. A connected set with no edge of $E_i$ is connected in $T-E_i$ and therefore lies in one member of $\mathcal P_i$. Hence a solution is a rainbow partition subordinate to $(\mathcal P_i)$.

For the covering statement, let $K_i$ be the component of $T-E_i$ containing $B_i$ if $B_i\ne\varnothing$; such a component exists because $B_i$ is connected in $T-E_i$ by~\eqref{eq:avoid}. If $B_i=\varnothing$, let $K_i$ be any component of $T-E_i$. Then $V(T)=\bigcup_iB_i\subseteq\bigcup_iK_i$.
\end{proof}

\subsection{Normalization}
For an instance $(T,(E_c)_{c\in C})$, let $\col(e)=\{c\in C:e\in E_c\}$. An edge $e$ is \emph{uncoloured} if $\col(e)=\varnothing$ and \emph{singly coloured} if $|\col(e)|=1$. An instance is \emph{saturated} if every edge is singly coloured and $|E_c|=k-1$ for every $c\in C$. A saturated instance has
\[
 |E(T)|=k(k-1)\quad\text{and}\quad |V(T)|=k(k-1)+1.
\]

\begin{lemma}[Normalization]
\label{lem:normalization}
Let $k\ge1$, let $C$ be a set of $k$ colours, and let $(T,(E_c)_{c\in C})$ be an admissible instance.
\begin{enumerate}
\item There is an admissible instance $(\widehat T,(\widehat E_c)_{c\in C})$, computable in polynomial time, in which every edge is singly coloured and $|\widehat E_c|=|E_c|$ for every $c\in C$, such that from every solution of it a solution of $(T,(E_c)_{c\in C})$ can be computed in polynomial time.
\item There is a saturated instance $(T^+,(E^+_c)_{c\in C})$ containing $\widehat T$ as an induced subtree such that every solution of it restricts to a solution of $(\widehat T,(\widehat E_c)_{c\in C})$. In particular, every solution of $(T^+,(E^+_c)_{c\in C})$ yields a solution of $(T,(E_c)_{c\in C})$.
\end{enumerate}
\end{lemma}

\begin{proof}
We apply three transformations. Each preserves the colour set $C$ and admissibility, and we show how solutions are transferred back.

\emph{Contraction.} Let $\mathcal U$ be the set of components of the spanning subgraph of $T$ formed by the uncoloured edges. This is a connected partition of $V(T)$. A coloured edge cannot have both endpoints in one member of $\mathcal U$, since together with an uncoloured path between them it would form a cycle. Thus the edges joining distinct members of $\mathcal U$ are exactly the coloured edges, and by Lemma~\ref{lem:tree}(iii) the quotient $T'=T/\mathcal U$ is a tree whose edges correspond bijectively to the coloured edges of $T$. Give every edge of $T'$ the colour set of the corresponding edge of $T$; then $|E'_c|=|E_c|$. Let $(B'_c)$ be a solution for $T'$ and put $B_c=\pi_{\mathcal U}^{-1}(B'_c)$. These sets form an allocation of $V(T)$, and nonempty ones are connected by Lemma~\ref{lem:tree}(iii). If an edge $uv$ of $T[B_c]$ belonged to $E_c$, it would be coloured, so $\pi_{\mathcal U}(u)\ne\pi_{\mathcal U}(v)$ would be joined by an edge of $E'_c$ inside $B'_c$, contrary to~\eqref{eq:avoid}.

\emph{Subdivision.} Now every edge is coloured. Replace each edge $uv$ with $\col(uv)=\{c_1,\ldots,c_t\}$, where $t\ge1$, by a path $u=x_0,x_1,\ldots,x_t=v$ through $t-1$ new vertices, and give the edge $x_{j-1}x_j$ the single colour $c_j$. Different edges receive disjoint sets of new vertices. The resulting graph $\widehat T$ is connected and has one edge fewer than the number of vertices, so it is a tree, and every colour occurs on as many edges as before. Let $(\widehat B_c)$ be a solution for $\widehat T$ and put $B_c=\widehat B_c\cap V(T)$. These sets form an allocation of $V(T)$. Every new vertex has degree $2$, and its neighbours are its predecessor and successor on the replacement path that contains it. Hence, a path of $\widehat T$ between old vertices that enters the interior of the replacement path of an old edge $uv$ traverses the whole replacement path from $u$ to $v$ or from $v$ to $u$. An edge of $\widehat T$ with two old endpoints is an old edge. Consequently, the old vertices on the path of $\widehat T$ between two old vertices $x,y$, taken in order, form the $x$--$y$ path of $T$. If $x,y\in B_c$, the path of $\widehat T$ between them lies in $\widehat B_c$ by Lemma~\ref{lem:tree}(i), so the $x$--$y$ path of $T$ lies in $B_c$; thus nonempty sets $B_c$ are connected. If an old edge $uv$ with $c\in\col(uv)$ had both endpoints in $B_c$, then its entire replacement path, which contains an edge of colour $c$, would lie in $\widehat B_c$, contrary to~\eqref{eq:avoid}.

\emph{Saturation.} For every colour $c$ with $|\widehat E_c|<k-1$, attach $k-1-|\widehat E_c|$ new pendant edges of colour $c$ to arbitrary vertices, each with a new leaf. The result $T^+$ is a tree containing $\widehat T$ as an induced subtree, and the new instance is saturated. For a solution $(B^+_c)$ of the new instance, the sets $B^+_c\cap V(\widehat T)$ form an allocation of $V(\widehat T)$. They are connected or empty by Lemma~\ref{lem:tree}(i), and they satisfy~\eqref{eq:avoid} because $\widehat T$ is an induced subgraph of $T^+$.

The first two steps, whose transfers compose, prove~(i); these steps and the transfers are clearly polynomial. The third step proves~(ii).
\end{proof}

\subsection{Nonempty pieces}

\begin{proposition}
\label{prop:nonempty}
Let $k\ge1$ and let $T$ be a tree with $|V(T)|\ge k$. If a rainbow partition subordinate to connected partitions $\mathcal P_1,\ldots,\mathcal P_k$ of $V(T)$ exists, then one with all pieces nonempty exists. The same holds for solutions of an instance of the edge-colour form.
\end{proposition}

\begin{proof}
Suppose that $B_i=\varnothing$. At most $k-1$ nonempty pieces cover at least $k$ vertices, so some piece $B_j$ has at least two vertices. The tree $T[B_j]$ has a leaf $v$. Replace $B_j$ by $B_j\setminus\{v\}$, which is connected and contained in the member of $\mathcal P_j$ containing $B_j$, and replace $B_i$ by $\{v\}$, which is contained in the member of $\mathcal P_i$ containing $v$. For the edge-colour form, note that $\{v\}$ induces no edge and that $B_j\setminus\{v\}\subseteq B_j$. The number of nonempty pieces increases, and the iteration completes the proof.
\end{proof}

\section{From inward partitions to rainbow partitions}
\label{sec:reduction}

In this section, we derive Theorems~\ref{thm:A} and~\ref{thm:Aprime} from Theorem~\ref{thm:C}, which is proved in Section~\ref{sec:orientation}.

\subsection{Balanced splittings}
For a function $\mu\colon W\to\R$ and $X\subseteq W$, write $\mu(X)=\sum_{v\in X}\mu(v)$.

\begin{corollary}[Balanced splitting]
\label{cor:split}
Let $Q=(W,D)$ be a tree with $s\ge2$ vertices, and let $\mu_1,\ldots,\mu_{s-1}\colon W\to\R$ satisfy $\mu_c(W)=0$ for every $c$. Then there are a connected partition $\{Z_1,\ldots,Z_r\}$ of $W$ with $r\ge2$ and a partition $\{0,1,\ldots,s-1\}=G_1\cup\cdots\cup G_r$ into pairwise disjoint sets such that $|G_j|=|Z_j|$ for every $j$ and
\[
 \mu_c(Z_j)\ge0\qquad(j\in[r],\ c\in G_j\setminus\{0\}).
\]
\end{corollary}

\begin{proof}
For $c\in[s-1]$, define an orientation $\omega_c$ as follows. If $f\in D$ has sides $A$ and $B$, that is, $\Comp(Q-f)=\{A,B\}$, where $f$ has one endpoint in each side by Lemma~\ref{lem:tree}(iv), then $\mu_c(A)+\mu_c(B)=0$, so at least one side has nonnegative $\mu_c$-mass. Direct $f$ towards its endpoint in such a side, choosing either side if both masses vanish. Thus, the head of every edge lies on a side of nonnegative mass.

Let $Z$ be an $\omega_c$-inward set, and let $H\in\Comp(Q-Z)$. By Lemma~\ref{lem:tree}(ii), $H$ is the side of $f_H$ not containing its endpoint $q\in Z$, and the other side is $W\setminus H$. Since $f_H$ is directed towards $q$, we get $\mu_c(W\setminus H)\ge0$, hence $\mu_c(H)\le0$. The components of $Q-Z$ partition $W\setminus Z$, so
\[
 \mu_c(Z)=\mu_c(W)-\sum_{H\in\Comp(Q-Z)}\mu_c(H)\ge0 .
\]
Applying Theorem~\ref{thm:C} to $\omega_1,\ldots,\omega_{s-1}$ gives the required partitions.
\end{proof}

\begin{corollary}[Integer splitting]
\label{cor:integer}
Let $Q=(W,D)$ be a tree with $s\ge2$ vertices, and let $e_1,\ldots,e_{s-1}\colon W\to\mathbb Z_{\ge0}$ satisfy $e_c(W)=s-1$ for every $c$. Then there are partitions as in Corollary~\ref{cor:split} such that
\[
 e_c(Z_j)\le|Z_j|-1\qquad(j\in[r],\ c\in G_j\setminus\{0\}).
\]
\end{corollary}

\begin{proof}
Apply Corollary~\ref{cor:split} to $\mu_c(v)=1-e_c(v)-1/s$, for which $\mu_c(W)=s-(s-1)-1=0$. If $\mu_c(Z)\ge0$, then the integer $|Z|-e_c(Z)$ is at least $|Z|/s>0$, hence at least $1$.
\end{proof}

\begin{remark}
If $\nu_1,\ldots,\nu_{s-1}$ are probability distributions on $W$, then Corollary~\ref{cor:split} with \mbox{$\mu_c(v)=1/s-\nu_c(v)$} yields a connected partition with $r\ge2$ parts and a grouping of the indices with $|G_j|=|Z_j|$ in which $\nu_c(Z_j)\le|Z_j|/s$ for all $c\in G_j\setminus\{0\}$; with $\mu_c(v)=\nu_c(v)-1/s$ one obtains $\nu_c(Z_j)\ge|Z_j|/s$ instead. In words, the vertices can be split into $r\ge2$ connected regions and the indices into groups of the same sizes, so that every index $c\ne0$ in a group of size $m$ gives the region of its group measure at most (or, in the second version, at least) $m/s$, the proportional share of the whole group.
\end{remark}

\subsection{Proof of the rainbow partition theorem}

\begin{proof}[Proof of Theorem~\ref{thm:Aprime}]
We prove, by strong induction on $k$, that for every tree $T$, every set $C$ of $k$ colours, and all edge sets $E_c\subseteq E(T)$ with $|E_c|\le k-1$, the instance $(T,(E_c)_{c\in C})$ has a solution. For $k=1$, the only edge set is empty, and $B_c=V(T)$ is a solution.

Let $k\ge2$ and assume the statement for all smaller numbers of colours. By Lemma~\ref{lem:normalization}, we may assume that the instance is saturated. Fix a colour and denote it by $0$. Let $\mathcal H=\Comp(T-E_0)$, and let $\pi\colon V(T)\to\mathcal H$ be the quotient map. Since $|E_0|=k-1$, there are exactly $k$ members of $\mathcal H$, and by Lemma~\ref{lem:tree}(iii),(iv) the quotient $Q=T/\mathcal H$ is a tree on $k$ vertices whose edges correspond bijectively to $E_0$.

Every edge has exactly one colour. Hence, an edge of colour $c\ne0$ does not belong to $E_0$ and has both endpoints in one member of $\mathcal H$. For $c\in C\setminus\{0\}$ and $H\in\mathcal H$ put
\begin{equation}
 e_c(H)=|E_c\cap E(T[H])|,
 \qquad\text{so that}\qquad
 \sum_{H\in\mathcal H}e_c(H)=|E_c|=k-1.
 \label{eq:counts}
\end{equation}
Apply Corollary~\ref{cor:integer} to $Q$, with $s=k$, the colours of $C\setminus\{0\}$ in the role of $[s-1]$, and the colour $0$ in the role of the free index. We obtain a connected partition $\{Z_1,\ldots,Z_r\}$ of $V(Q)=\mathcal H$ with $r\ge2$ and a partition $C=G_1\cup\cdots\cup G_r$ with $|G_j|=|Z_j|$ and $e_c(Z_j)\le|Z_j|-1$ for all $c\in G_j\setminus\{0\}$. Put
\[
 U_j=\pi^{-1}(Z_j)=\bigcup_{H\in Z_j}H .
\]
By Lemma~\ref{lem:tree}(iii), the sets $U_1,\ldots,U_r$ are connected and form a partition of $V(T)$, and $T_j=T[U_j]$ is a tree.

Consider the instance on $T_j$ with colour set $G_j$ and edge sets $E_c\cap E(T_j)$, $c\in G_j$; colours outside $G_j$ are ignored. Let $c\in G_j\setminus\{0\}$. Every edge of colour $c$ in $T_j$ has both endpoints in a member $H$ of $\mathcal H$, and $H\subseteq U_j$ because $U_j$ is a union of members of $\mathcal H$; thus $H\in Z_j$. Conversely, every edge of $T[H]$ with $H\in Z_j$ is an edge of $T_j$. Therefore
\[
 |E_c\cap E(T_j)|=e_c(Z_j)\le|Z_j|-1=|G_j|-1.
\]
If $0\in G_j$, then the edges of colour $0$ in $T_j$ join distinct members of $\mathcal H$ lying in $Z_j$, and they correspond bijectively to the edges of the tree $Q[Z_j]$. Hence
\[
 |E_0\cap E(T_j)|=|Z_j|-1=|G_j|-1.
\]
Thus, every local instance is admissible with respect to its own colour set. Since $r\ge2$ and all parts are nonempty, $1\le|G_j|\le k-1$, and the induction hypothesis gives a solution $(B_c)_{c\in G_j}$ of the instance on $T_j$.

Every colour belongs to exactly one group, so the sets $B_c$, $c\in C$, form an allocation of $V(T)$. Every nonempty $B_c$ is connected in $T_j$, hence in $T$, and since $T_j$ is an induced subgraph of $T$, we have $T[B_c]=T_j[B_c]$, which gives~\eqref{eq:avoid}. This completes the induction. The covering statement follows from Proposition~\ref{prop:equivalence}.
\end{proof}

\begin{proof}[Proof of Theorem~\ref{thm:A}]
The first statement follows from Theorem~\ref{thm:Aprime} and Proposition~\ref{prop:equivalence}, and the statement about nonempty pieces from Proposition~\ref{prop:nonempty}.
\end{proof}

\begin{example}
\label{ex:reduction}
Consider the instance of Figure~\ref{fig:spider}, which is saturated for $k=3$. Take $a$ as the colour $0$. The components of $T-E_a$ are $H_1=\{y_1,z_1\}$, $H_2=\{x,y_2,z_2,y_3\}$, and $H_3=\{z_3\}$, and the quotient $Q$ is the path $H_1H_2H_3$. The counts~\eqref{eq:counts} are $(e_b(H_1),e_b(H_2),e_b(H_3))=(1,1,0)$ and $(e_c(H_1),e_c(H_2),e_c(H_3))=(0,2,0)$. The orientations constructed in the proofs of Corollaries~\ref{cor:split} and~\ref{cor:integer} are $H_1\to H_2\to H_3$ for $b$ and $H_1\leftarrow H_2\to H_3$ for $c$. The partition $\bigl\{\{H_1\},\{H_2,H_3\}\bigr\}$ of $V(Q)$ with groups $\{c\}$ and $\{a,b\}$ satisfies Theorem~\ref{thm:C}: the set $\{H_1\}$ is inward for the orientation of $c$, and $\{H_2,H_3\}$ is inward for the orientation of $b$. The local instances are $T[\{y_1,z_1\}]$ with colour $c$, which contains no edge of colour $c$, and $T[\{x,y_2,z_2,y_3,z_3\}]$ with colours $a$ and $b$, each occurring once. Choosing the solutions $B_c=\{y_1,z_1\}$ and $B_a=\{x,y_2,z_2\}$, $B_b=\{y_3,z_3\}$ of the local instances gives the rainbow partition of Figure~\ref{fig:spider}(a); see Figure~\ref{fig:reduction}.
\end{example}

\begin{figure}[t]
\centering
\begin{tikzpicture}[scale=0.66,
  vtx/.style={circle,fill=black,inner sep=1.1pt},
  lab/.style={inner sep=0.6pt,font=\scriptsize},
  qn/.style={circle,draw,minimum size=17pt,inner sep=0pt,font=\scriptsize}]
\begin{scope}
\coordinate (x) at (0,0);
\coordinate (y1) at (90:1.15);  \coordinate (z1) at (90:2.3);
\coordinate (y2) at (210:1.15); \coordinate (z2) at (210:2.3);
\coordinate (y3) at (330:1.15); \coordinate (z3) at (330:2.3);
\draw[line width=15pt,line cap=round,gray!25] (y1)--(z1);
\draw[line width=15pt,line cap=round,line join=round,gray!25] (z2)--(x)--(y3);
\fill[gray!25] (z3) circle (7.5pt);
\draw[very thick] (x)--(y1) node[lab,midway,left=1pt]{$a$};
\draw (y1)--(z1) node[lab,midway,left=1pt]{$b$};
\draw (x)--(y2) node[lab,midway,above=1pt]{$b$};
\draw (y2)--(z2) node[lab,midway,above=1pt]{$c$};
\draw (x)--(y3) node[lab,midway,above=1pt]{$c$};
\draw[very thick] (y3)--(z3) node[lab,midway,above=1pt]{$a$};
\foreach \v in {x,y1,z1,y2,z2,y3,z3} \node[vtx] at (\v) {};
\node[font=\scriptsize] at ($(z1)+(0.75,0)$) {$H_1$};
\node[font=\scriptsize] at ($(z2)+(0,-0.6)$) {$H_2$};
\node[font=\scriptsize] at ($(z3)+(0,-0.6)$) {$H_3$};
\node at (0,-3.0) {\small (a)};
\end{scope}
\begin{scope}[xshift=5.2cm,yshift=0.9cm]
\node[font=\scriptsize,anchor=east] at (-0.4,0) {$b$:};
\node[qn] (b1) at (0.5,0) {$1$}; \node[qn] (b2) at (2.3,0) {$1$}; \node[qn] (b3) at (4.1,0) {$0$};
\draw[-{Stealth[length=5pt]},thick] (b1)--(b2); \draw[-{Stealth[length=5pt]},thick] (b2)--(b3);
\node[font=\scriptsize,anchor=east] at (-0.4,-1.6) {$c$:};
\node[qn] (c1) at (0.5,-1.6) {$0$}; \node[qn] (c2) at (2.3,-1.6) {$2$}; \node[qn] (c3) at (4.1,-1.6) {$0$};
\draw[-{Stealth[length=5pt]},thick] (c2)--(c1); \draw[-{Stealth[length=5pt]},thick] (c2)--(c3);
\node[font=\scriptsize] at (0.5,0.62) {$H_1$}; \node[font=\scriptsize] at (2.3,0.62) {$H_2$}; \node[font=\scriptsize] at (4.1,0.62) {$H_3$};
\draw[densely dashed] (1.4,0.75)--(1.4,-2.3);
\node[font=\scriptsize] at (0.5,-2.45) {$\{c\}$}; \node[font=\scriptsize] at (3.2,-2.45) {$\{a,b\}$};
\node at (2.3,-3.9) {\small (b)};
\end{scope}
\begin{scope}[xshift=13.8cm]
\coordinate (x) at (0,0);
\coordinate (y1) at (90:1.15);  \coordinate (z1) at (90:2.3);
\coordinate (y2) at (210:1.15); \coordinate (z2) at (210:2.3);
\coordinate (y3) at (330:1.15); \coordinate (z3) at (330:2.3);
\draw[line width=13pt,line cap=round,red!22] (x)--(z2);
\draw[line width=13pt,line cap=round,blue!22] (y3)--(z3);
\draw[line width=13pt,line cap=round,green!30] (y1)--(z1);
\draw[densely dashed,gray!80] (x)--(y1);
\draw (y1)--(z1) node[lab,midway,left=1pt]{$b$};
\draw (x)--(y2) node[lab,midway,above=1pt]{$b$};
\draw (y2)--(z2) node[lab,midway,above=1pt]{$c$};
\draw (x)--(y3) node[lab,midway,above=1pt]{$c$};
\draw (y3)--(z3) node[lab,midway,above=1pt]{$a$};
\foreach \v in {x,y1,z1,y2,z2,y3,z3} \node[vtx] at (\v) {};
\node[font=\scriptsize] at ($(z1)+(0.75,0)$) {$B_c$};
\node[font=\scriptsize] at ($(z2)+(0,-0.6)$) {$B_a$};
\node[font=\scriptsize] at ($(z3)+(0,-0.6)$) {$B_b$};
\node at (0,-3.0) {\small (c)};
\end{scope}
\end{tikzpicture}
\caption{The reduction of Example~\ref{ex:reduction}. (a) The edges of colour $a$ (thick) and the components $H_1,H_2,H_3$ of $T-E_a$ (shaded). (b) The quotient path $H_1H_2H_3$ with the counts $e_b$ and $e_c$ written in the nodes, the orientations for $b$ and $c$, and the splitting $\{H_1\}\mid\{H_2,H_3\}$ with groups $\{c\}$ and $\{a,b\}$. (c) The local instances on $\{y_1,z_1\}$ and on $\{x,y_2,z_2,y_3,z_3\}$ and their solutions. Vertices are placed and named as in Figure~\ref{fig:spider}; in~(c) the dashed edge $xy_1$ is the edge of colour $a$ joining the two local instances.}
\label{fig:reduction}
\end{figure}
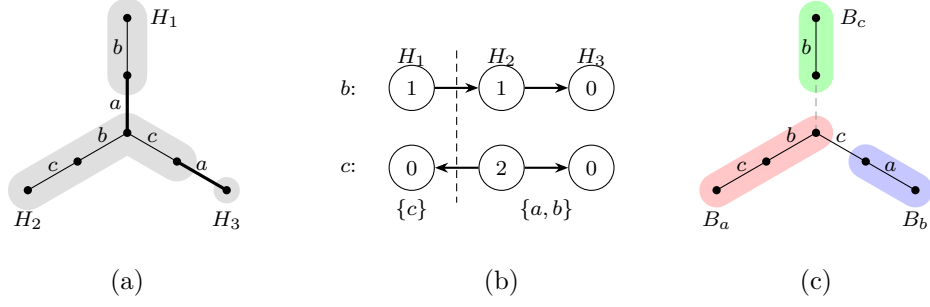

\section{Inward partitions of oriented trees}
\label{sec:orientation}

Throughout this section, $Q=(W,D)$ is a tree with $s\ge2$ vertices.

\subsection{Weighted reversal costs}
Let $\omega$ be an orientation of $Q$. For $f\in D$, let $H_\omega(f)$ be the component of $Q-f$ containing the tail of $f$. For $w=(w_f)_{f\in D}\in\R_{\ge0}^D$ and $y\in W$, define
\begin{equation}
 \phi_w(\omega,y)=\sum_{f\in D}w_f\,\ind{y\in H_\omega(f)},
 \qquad
 M_w(\omega)=\argmin_{y\in W}\phi_w(\omega,y),
 \label{eq:cost}
\end{equation}
where $\ind{\cdot}$ is the indicator function. An edge $f$ points away from $y$ exactly when $y$ lies on the tail side of $f$. Thus $\phi_w(\omega,y)$ is the total weight of the edges that must be reversed to direct every edge of $Q$ towards $y$. The set $M_w(\omega)$ is nonempty. Let $\supp(w)=\{f\in D:w_f>0\}$, and let
\[
 \Delta_D=\Bigl\{w\in\R_{\ge0}^D:\sum_{f\in D}w_f=1\Bigr\}
\]
be the simplex of normalised weights.

\begin{lemma}
\label{lem:min}
Let $\omega$ be an orientation of $Q$ and let $w\in\R_{\ge0}^D$.
\begin{enumerate}
\item If $f\in D$ is directed from $u$ to $v$, then $\phi_w(\omega,v)=\phi_w(\omega,u)-w_f$.
\item Let $J=\supp(w)$ and $P\in\CC J$. If $P\cap M_w(\omega)\ne\varnothing$, then $P\subseteq M_w(\omega)$ and $P$ is $\omega$-inward.
\item For every $y\in W$, the set $\{w\in\Delta_D:y\in M_w(\omega)\}$ is closed.
\end{enumerate}
\end{lemma}

\begin{proof}
(i) The vertices $u$ and $v$ lie on the same side of every edge $g\ne f$, and $u\in H_\omega(f)$ while $v\notin H_\omega(f)$.

(ii) The internal edges of $P$ lie outside $J$ by Lemma~\ref{lem:tree}(iv), so they have weight zero, and by~(i) and the connectivity of $P$, the function $\phi_w(\omega,\cdot)$ is constant on $P$. Since $P$ meets $M_w(\omega)$, we get $P\subseteq M_w(\omega)$. Every edge of $\bd P$ lies in $J$ and has positive weight. If such an edge were directed from $u\in P$ to $v\notin P$, then~(i) would give $\phi_w(\omega,v)<\phi_w(\omega,u)$, contradicting $u\in M_w(\omega)$.

(iii) The condition $y\in M_w(\omega)$ means that $\phi_w(\omega,y)\le\phi_w(\omega,u)$ for all $u\in W$, a finite system of non-strict linear inequalities in $w$.
\end{proof}

\subsection{The common-weight theorem}

\begin{theorem}[Common weights]
\label{thm:weights}
Let $\omega_1,\ldots,\omega_{s-1}$ be orientations of $Q$, not necessarily distinct. There exist $w^*\in\Delta_D$ and an injective map $\beta\colon[s-1]\to W$ such that
\[
 \beta(c)\in M_{w^*}(\omega_c)\qquad(c\in[s-1]).
\]
\end{theorem}

\begin{proof}[Proof of Theorem~\ref{thm:C} from Theorem~\ref{thm:weights}]
Let $w^*$ and $\beta$ be as in Theorem~\ref{thm:weights}. The set $F=\supp(w^*)$ is nonempty, so the $F$-parts $Z_1,\ldots,Z_r$ form a connected partition of $W$ with $r=|F|+1\ge2$ (Lemma~\ref{lem:tree}(iv)). Since $|W|=s$, exactly one vertex is not in the image of $\beta$; extend $\beta$ to a bijection $\gamma\colon\{0,1,\ldots,s-1\}\to W$ by sending $0$ to this vertex, and put $G_j=\gamma^{-1}(Z_j)$. Then $|G_j|=|Z_j|$. If $c\in G_j\setminus\{0\}$, then $\gamma(c)=\beta(c)\in Z_j\cap M_{w^*}(\omega_c)$, and Lemma~\ref{lem:min}(ii) shows that $Z_j$ is $\omega_c$-inward.
\end{proof}

\begin{remark}
\label{rem:C-sharp}
(a) \emph{The bound $s-1$ and the free index.} If Theorem~\ref{thm:C} were stated for $s$ orientations $\omega_1=\cdots=\omega_s=\omega$, with $\{0,\ldots,s-1\}$ replaced by $[s]$, then every part $Z_j$ would receive at least one index and would have to be $\omega$-inward. This is impossible for $r\ge2$: an edge of $Q$ joining two parts is directed away from one of them. Similarly, Theorem~\ref{thm:weights} fails for $s$ orientations equal to $\omega$, since an injection $[s]\to M_w(\omega)$ forces $M_w(\omega)=W$, whereas $u\notin M_w(\omega)$ whenever an edge of positive weight is directed away from $u$, by Lemma~\ref{lem:min}(i).

(b) \emph{Vanishing weights are necessary.} If $s\ge3$ and every $\omega_c$ directs all edges towards a common root, then for weights that are positive on all edges, the root is the only minimiser for every orientation. Thus, the weights $w^*$ in Theorem~\ref{thm:weights} must, in general, vanish on some edges, and the parts obtained in the proof of Theorem~\ref{thm:C} are the components of the spanning subgraph of $Q$ formed by the edges of weight zero.

(c) \emph{Rental harmony.} Theorem~\ref{thm:weights} has the flavour of a rental-harmony theorem~\cite{AsadaEtAl2018,FrickEtAl2019,Su1999}: the vertices of $Q$ are rooms, the orientations are $s-1$ tenants, a weight vector plays the role of a price vector, and each tenant prefers the rooms minimising $\phi_w$. However, the parameter space $\Delta_D$ has dimension $s-2$ rather than $s-1$, and the orientations determine the preferences. We do not know how to derive Theorem~\ref{thm:weights} directly from rental-harmony theorems.
\end{remark}

\subsection{A labelling compatible with coarsening}
\label{sec:labels}
The constructions in this subsection depend only on the tree $Q$, not on orientations.

\parhead{Leaf-elimination order.}
Put $Q_1=Q$. For $t=1,\ldots,s-1$, the graph $Q_t$ is a tree with $s-t+1\ge2$ vertices; choose a leaf $\ell_t$ of $Q_t$, let $f_t=\ell_tp_t$ be the edge of $Q_t$ at $\ell_t$, and put $Q_{t+1}=Q_t-\ell_t$. Then $f_1,\ldots,f_{s-1}$ is an enumeration of $D$,
\begin{equation}
 E(Q_t)=\{f_t,f_{t+1},\ldots,f_{s-1}\},
 \qquad\text{and}\qquad
 p_t\in V(Q_{t+1}).
 \label{eq:leaf-order}
\end{equation}
We order $D$ by $f_1<f_2<\cdots<f_{s-1}$. For nonempty $J\subseteq D$, $\max J$ always refers to this order.

For nonempty $F\subseteq D$ with $\max F=f_t$, let
\[
 Z^*(F)=Z_F(p_t)
 \qquad\text{and}\qquad
 Z^\circ(F)=Z_F(\ell_t).
\]
These two $F$-parts are distinct and joined by $f_t$, by Lemma~\ref{lem:tree}(iv).

\begin{lemma}[Largest edges]
\label{lem:largest}
If $\varnothing\ne F\subseteq J\subseteq D$, then $\max J=\max F$ or $\max J\in\inn{Z^*(F)}$.
\end{lemma}

\begin{proof}
Let $f_t=\max F$ and $f_u=\max J$. Then $u\ge t$ because $F\subseteq J$; assume $u>t$. The tree $Q_{t+1}$ contains $p_t$, and its edges $f_{t+1},\ldots,f_{s-1}$ do not belong to $F$. Hence $V(Q_{t+1})$ is contained in one $F$-part, namely $Z^*(F)$. Since $f_u\in E(Q_{t+1})$ by~\eqref{eq:leaf-order}, both endpoints of $f_u$ lie in $Z^*(F)$.
\end{proof}

\parhead{Designated boundary edges and slot sets.}
Let $\varnothing\ne F\subseteq D$ and $Z\in\CC F$. Define an edge $b_F(Z)$ as follows. If $Z=Z^*(F)$, put $b_F(Z)=\max F$. If $Z\ne Z^*(F)$, then $Z^*(F)$ is connected and disjoint from $Z$, so it lies in a component $H$ of $Q-Z$, and we put $b_F(Z)=f_H$, the edge joining $Z$ to $H$ (Lemma~\ref{lem:tree}(ii)); informally, $b_F(Z)$ is the first edge on the way from $Z$ towards $Z^*(F)$. In both cases $b_F(Z)\in\bd Z$. Moreover, $b_F(Z^\circ(F))=\max F$: the component of $Q-Z^\circ(F)$ containing $Z^*(F)$ contains $p_t$, and $f_t=\ell_tp_t$ joins it to $Z^\circ(F)$. The \emph{slot set} of $Z$ is
\begin{equation}
 S_F(Z)=\inn Z\cup\{b_F(Z)\}.
 \label{eq:slots}
\end{equation}
Since $Q[Z]$ is a tree and $b_F(Z)$ is a boundary edge,
\begin{equation}
 |S_F(Z)|=(|Z|-1)+1=|Z| .
 \label{eq:capacity}
\end{equation}
Slot sets of different $F$-parts need not be disjoint; only~\eqref{eq:capacity} will be used.

\parhead{The label of a part.}
Let $\varnothing\ne J\subseteq D$ and $P\in\CC J$, and write $f^*=\max J=ab$. Since $P$ is connected in $Q-J$, which is a subgraph of $Q-f^*$, the set $P$ lies in one component of $Q-f^*$; we name the endpoints of $f^*$ so that $a$ lies in this component. Define
\[
 \Lambda(P,J)=
 \begin{cases}
  f^* & \text{if } a\in P,\\
  f_H & \text{if } a\notin P,\ \text{where $H$ is the component of $Q-P$ containing $a$.}
 \end{cases}
\]
In the second case, $\Lambda(P,J)$ is the first edge of the path from $P$ to $a$ (Lemma~\ref{lem:tree}(ii)). In both cases
\begin{equation}
 \Lambda(P,J)\in\bd P\subseteq J .
 \label{eq:label-support}
\end{equation}
In terms of the quotient tree $Q/\CC J$, whose edges are identified with $J$, the label of the node $P$ is the first edge on the path from $P$ towards the edge $\max J$, and it is $\max J$ itself if $P$ is incident to $\max J$.

\begin{lemma}[Compatibility under coarsening]
\label{lem:compatibility}
Let $\varnothing\ne F\subseteq J\subseteq D$, let $P\in\CC J$, and let $Z\in\CC F$ contain $P$. Then
\[
 \Lambda(P,J)\in S_F(Z).
\]
\end{lemma}

\begin{proof}
Write $f_t=\max F$, $Z^*=Z^*(F)$, $Z^\circ=Z^\circ(F)$, $f^*=\max J=ab$ with $a$ as in the definition of the label, and $\lambda=\Lambda(P,J)$.

\emph{Case 1: $a\in P$.} Then $\lambda=f^*$ and $a\in Z$. If $b\in Z$, then $\lambda\in\inn Z$. Otherwise $\lambda\in\bd Z\subseteq F$ by Lemma~\ref{lem:tree}(iv), and since $\lambda$ is the largest edge of $J\supseteq F$, we get $\lambda=f_t$. Then $Z$ contains an endpoint of $f_t$, so $Z\in\{Z^*,Z^\circ\}$, and in both cases $b_F(Z)=f_t=\lambda$.

\emph{Case 2: $a\notin P$.} Let $q$ be the vertex of $P$ closest to $a$. Then $\lambda$ is the first edge of the path from $q$ to $a$. If $a\in Z$, then this path lies in $Z$ by Lemma~\ref{lem:tree}(i), and $\lambda\in\inn Z$. Suppose that $a\notin Z$, and let $H'$ be the component of $Q-Z$ containing $a$. By Lemma~\ref{lem:tree}(ii), the path from $q$ to $a$ runs inside $Z$ until it leaves $Z$ through $f_{H'}$; hence $\lambda\in\inn Z$ or $\lambda=f_{H'}$. It therefore suffices to show that $Z\ne Z^*$ and $Z^*\subseteq H'$, since then $f_{H'}=b_F(Z)$.

By Lemma~\ref{lem:largest}, either $f^*\in\inn{Z^*}$ or $f^*=f_t$. If $f^*\in\inn{Z^*}$, or if $f^*=f_t$ and $a=p_t$, then $a\in Z^*$. As $a\notin Z$, we have $Z\ne Z^*$, and the connected set $Z^*$, being disjoint from $Z$, lies in the component $H'$ of $Q-Z$ containing $a$. It remains to consider $f^*=f_t$ and $a=\ell_t$. Then $a\in Z^\circ$, so $Z\ne Z^\circ$. By the choice of $a$, the set $P$ lies in the component of $Q-f_t$ containing $\ell_t$. Since $f_t\in F$, the $F$-part $Z$ is connected in $Q-f_t$, and as $P\subseteq Z$, it lies in the same component, whereas $Z^*\ni p_t$ lies in the other. Hence $Z\ne Z^*$. Finally, $Z^\circ\cup Z^*$ is connected (its two parts are joined by $f_t$), disjoint from $Z$, and contains $a$. It therefore lies in $H'$, and so $Z^*\subseteq H'$.
\end{proof}

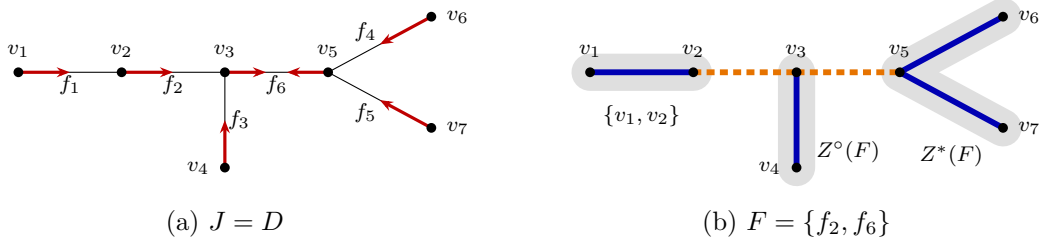
\begin{figure}[t]
\centering
\begin{tikzpicture}[scale=1.05,
  vtx/.style={circle,fill=black,inner sep=1.3pt},
  idx/.style={font=\scriptsize,inner sep=1pt},
  lbl/.style={-{Stealth[length=5pt]},very thick,red!75!black}]
\newcommand{\labtree}{%
  \coordinate (a) at (0,0); \coordinate (b) at (1.3,0); \coordinate (c) at (2.6,0);
  \coordinate (d) at (2.6,-1.2); \coordinate (e) at (3.9,0);
  \coordinate (f) at (5.2,0.7); \coordinate (g) at (5.2,-0.7);}
\newcommand{\labnames}{%
  \node[above=2pt,font=\scriptsize] at (a) {$v_1$}; \node[above=2pt,font=\scriptsize] at (b) {$v_2$}; \node[above=2pt,font=\scriptsize] at (c) {$v_3$};
  \node[left=2pt,font=\scriptsize] at (d) {$v_4$}; \node[above=2pt,font=\scriptsize] at (e) {$v_5$};
  \node[right=2pt,font=\scriptsize] at (f) {$v_6$}; \node[right=2pt,font=\scriptsize] at (g) {$v_7$};}
\begin{scope}
\labtree
\draw (a)--(b) node[idx,midway,below]{$f_1$};
\draw (b)--(c) node[idx,midway,below]{$f_2$};
\draw (c)--(d) node[idx,midway,right]{$f_3$};
\draw (e)--(f) node[idx,midway,above left]{$f_4$};
\draw (e)--(g) node[idx,midway,below left]{$f_5$};
\draw (c)--(e) node[idx,midway,below]{$f_6$};
\draw[lbl] (a)--($(a)!0.5!(b)$);
\draw[lbl] (b)--($(b)!0.5!(c)$);
\draw[lbl] (c)--($(c)!0.4!(e)$);
\draw[lbl] (d)--($(d)!0.5!(c)$);
\draw[lbl] (e)--($(e)!0.4!(c)$);
\draw[lbl] (f)--($(f)!0.5!(e)$);
\draw[lbl] (g)--($(g)!0.5!(e)$);
\foreach \v in {a,b,c,d,e,f,g} \node[vtx] at (\v) {};
\labnames
\node at (2.6,-1.9) {\small (a) $J=D$};
\end{scope}
\begin{scope}[xshift=7.2cm]
\labtree
\draw[line width=14pt,line cap=round,gray!25] (a)--(b);
\draw[line width=14pt,line cap=round,gray!25] (c)--(d);
\draw[line width=14pt,line cap=round,line join=round,gray!25] (f)--(e)--(g);
\draw[line width=2.2pt,blue!70!black] (a)--(b);
\draw[line width=2.2pt,blue!70!black] (c)--(d);
\draw[line width=2.2pt,blue!70!black] (e)--(f);
\draw[line width=2.2pt,blue!70!black] (e)--(g);
\draw[line width=2.2pt,orange!90!black,densely dashed] (b)--(c);
\draw[line width=2.2pt,orange!90!black,densely dashed] (c)--(e);
\foreach \v in {a,b,c,d,e,f,g} \node[vtx] at (\v) {};
\labnames
\node[font=\scriptsize] at (0.65,-0.55) {$\{v_1,v_2\}$};
\node[font=\scriptsize] at (3.25,-1.0) {$Z^\circ(F)$};
\node[font=\scriptsize] at (4.55,-1.05) {$Z^*(F)$};
\node at (2.6,-1.9) {\small (b) $F=\{f_2,f_6\}$};
\end{scope}
\end{tikzpicture}
\caption{The labelling of Section~\ref{sec:labels} for the leaf-elimination order $f_1=v_1v_2$, $f_2=v_2v_3$, $f_3=v_3v_4$, $f_4=v_5v_6$, $f_5=v_5v_7$, $f_6=v_3v_5$, with $\ell_6=v_3$ and $p_6=v_5$. (a) For $J=D$, every part is a single vertex, and the arrow at a vertex indicates its label, the first edge towards $\max D=f_6$. (b) For $F=\{f_2,f_6\}$ the $F$-parts are $\{v_1,v_2\}$, $Z^\circ(F)=\{v_3,v_4\}$, and $Z^*(F)=\{v_5,v_6,v_7\}$; internal edges are drawn solid and the designated boundary edges $b_F(\{v_1,v_2\})=f_2$ and $b_F(Z^\circ(F))=b_F(Z^*(F))=f_6$ dashed. The slot sets are $\{f_1,f_2\}$, $\{f_3,f_6\}$, and $\{f_4,f_5,f_6\}$, and each contains the labels in~(a) of the vertices of its part, as asserted by Lemma~\ref{lem:compatibility}.}
\label{fig:labels}
\end{figure}

\begin{remark}
\label{rem:labels}
(a) Since $b_F(Z)\in F$ and $\inn Z\cap F=\varnothing$, Lemma~\ref{lem:compatibility} says that a label outside $F$ is an internal edge of $Z$, and a label in $F$ equals $b_F(Z)$. In the proof below, the labels outside $F$ correspond to weights that vanish in the limit.

(b) \emph{The leaf-elimination order is needed.} Let $Q$ be the path $v_1v_2v_3v_4$, and use the same definitions with the order $v_2v_3<v_1v_2<v_3v_4$, which is not a leaf-elimination order. Let $F=\{v_1v_2,v_2v_3\}$, so that $Z=\{v_2\}$ is an $F$-part, and $\{v_2\}$ is also a $D$-part. The label rule gives $\Lambda(\{v_2\},F)=v_1v_2$, because $\max F=v_1v_2$ is incident to $v_2$, and $\Lambda(\{v_2\},D)=v_2v_3$, because $\max D=v_3v_4$ and the path from $v_2$ to $v_3$ starts with $v_2v_3$. Two distinct labels thus fall into the part $Z=\{v_2\}$ of size one, so no choice of a single designated boundary edge can make Lemma~\ref{lem:compatibility} hold.

(c) \emph{Why the label depends only on the part and the support.} Suppose that each orientation $c$ could label its part $P_c(w)$ in the proof below by an edge of $\bd{P_c(w)}$ of its own choice. Let $Q$ be a star with centre $z$ and $s\ge3$ vertices, and let all $s-1$ orientations direct every edge towards $z$. For weights of full support, $z$ is then the unique minimiser for every orientation, and the $s-1$ orientations could carry the $s-1$ distinct edges at $z$ as labels. The limit argument below would then place $s-1$ orientations in the part $\{z\}$ of size one. The rule $\Lambda$ excludes this, because the label of a part depends only on the part and the support; all these orientations receive the label $\Lambda(\{z\},D)=\max D$. By~(b), independence of the orientation is not enough on its own, and the leaf-elimination order is needed as well.
\end{remark}

\subsection{The colourful KKM theorem and the limit argument}
For a finite set $I$, let
\[
 \Delta_I=\Bigl\{w\in\R^I_{\ge0}:\sum_{x\in I}w_x=1\Bigr\},
\]
let $\supp(w)=\{x\in I:w_x>0\}$ for $w\in\Delta_I$, and for nonempty $A\subseteq I$ let
\[
 \Delta_I(A)=\{w\in\Delta_I:\supp(w)\subseteq A\}
\]
be the face spanned by the unit vectors of $A$. A family $(C_x)_{x\in I}$ of closed subsets of $\Delta_I$ is a \emph{KKM cover} if $\Delta_I(A)\subseteq\bigcup_{x\in A}C_x$ for every nonempty $A\subseteq I$.

\begin{theorem}[Colourful KKM theorem, Gale~\cite{Gale1984}]
\label{thm:gale}
Let $I$ be a finite set with $|I|=L\ge1$, and for every $c\in[L]$ let $(C^c_x)_{x\in I}$ be a KKM cover of $\Delta_I$. Then there exist $w^*\in\Delta_I$ and a bijection $x\mapsto c_x$ from $I$ onto $[L]$ such that $w^*\in C^{c_x}_x$ for every $x\in I$.
\end{theorem}

For $L$ identical covers, this is the classical KKM theorem~\cite{KKM1929}. Bapat's permutation-based generalisation of Sperner's lemma~\cite{Bapat1989} is its combinatorial counterpart; see~\cite{AsadaEtAl2018} for further discussion and generalisations.

\begin{proof}[Proof of Theorem~\ref{thm:weights}]
Put $L=s-1=|D|$. Fix a leaf-elimination order, that is, a choice of $\ell_1,\ldots,\ell_{s-1}$ as in Section~\ref{sec:labels}, and an arbitrary linear order of $W$. For $c\in[L]$ and $w\in\Delta_D$, let $y_c(w)$ be the least element of $M_w(\omega_c)$, let $P_c(w)=Z_{\supp(w)}(y_c(w))$ be the $\supp(w)$-part containing it, and define
\[
 \lambda_c(w)=\Lambda\bigl(P_c(w),\supp(w)\bigr)\in\supp(w),
\]
where the membership holds by~\eqref{eq:label-support}. By Lemma~\ref{lem:min}(ii),
\begin{equation}
 P_c(w)\subseteq M_w(\omega_c).
 \label{eq:Pc-min}
\end{equation}
For $x\in D$, let $C^c_x$ be the closure of $\lambda_c^{-1}(x)$ in $\Delta_D$. For every $c$, the family $(C^c_x)_{x\in D}$ is a KKM cover: if $w\in\Delta_D(A)$, then $x=\lambda_c(w)$ lies in $\supp(w)\subseteq A$, and $w\in C^c_x$. Theorem~\ref{thm:gale} gives $w^*\in\Delta_D$ and a bijection $x\mapsto c_x$ from $D$ onto $[L]$ with $w^*\in C^{c_x}_x$ for all $x$. Let \mbox{$F=\supp(w^*)$}; then $F\ne\varnothing$.

Fix $x\in D$. Since $w^*$ lies in the closure of $\lambda_{c_x}^{-1}(x)$, there are $v^{(m)}\in\Delta_D$, $m\ge1$, with $\lambda_{c_x}(v^{(m)})=x$ and $v^{(m)}\to w^*$. The pair $\bigl(\supp(v^{(m)}),P_{c_x}(v^{(m)})\bigr)$ takes only finitely many values, so after passing to a subsequence, it is constant, equal to $(J_x,P_x)$, say. Then $x=\Lambda(P_x,J_x)$, and
\begin{equation}
 F\subseteq J_x,
 \label{eq:support-inclusion}
\end{equation}
because a positive coordinate of $w^*$ is positive in $v^{(m)}$ for all large $m$. Choose $y_x\in P_x$. By~\eqref{eq:Pc-min}, $y_x\in M_{v^{(m)}}(\omega_{c_x})$ for all $m$ in the subsequence, and Lemma~\ref{lem:min}(iii) yields $y_x\in M_{w^*}(\omega_{c_x})$. Let $Z_x=Z_F(y_x)$. By~\eqref{eq:support-inclusion} and Lemma~\ref{lem:tree}(iv), the $J_x$-part $P_x$ lies in an $F$-part, which must be $Z_x$. Lemma~\ref{lem:min}(ii), applied at $w^*$, gives
\begin{equation}
 Z_x\subseteq M_{w^*}(\omega_{c_x}),
 \label{eq:limit-parts}
\end{equation}
and Lemma~\ref{lem:compatibility}, applied to $F\subseteq J_x$ and $P_x\subseteq Z_x$, gives
\begin{equation}
 x\in S_F(Z_x).
 \label{eq:limit-labels}
\end{equation}

We now assign the orientations to the vertices. For $Z\in\CC F$, let $A_Z=\{c_x:x\in D,\ Z_x=Z\}$. Since $x\mapsto c_x$ is a bijection, every $c\in[L]$ lies in exactly one set $A_Z$, and by~\eqref{eq:limit-labels} and~\eqref{eq:capacity},
\[
 |A_Z|=|\{x\in D:Z_x=Z\}|\le|S_F(Z)|=|Z| .
\]
Choose an injection $\beta_Z\colon A_Z\to Z$ for every $Z\in\CC F$. Since the $F$-parts are pairwise disjoint, these injections combine into an injection $\beta\colon[L]\to W$. If $c=c_x$, then $\beta(c)\in Z_x\subseteq M_{w^*}(\omega_c)$ by~\eqref{eq:limit-parts}.
\end{proof}

\begin{remark}
(a) The argument uses no continuity of the minimisers $y_c(w)$, the parts $P_c(w)$, or the labels $\lambda_c(w)$. Of the dependence on $w$, it uses only the closedness in Lemma~\ref{lem:min}(iii) and the fact that supports cannot shrink near $w^*$, which gives~\eqref{eq:support-inclusion}, besides the combinatorial Lemmas~\ref{lem:min}(ii) and~\ref{lem:compatibility}.

(b) Theorem~\ref{thm:gale} can be replaced by Bapat's lemma~\cite{Bapat1989}. By~\eqref{eq:label-support}, the restrictions of $\lambda_1,\ldots,\lambda_L$ to the vertices of a triangulation of $\Delta_D$ are Sperner labellings, that is, every vertex $v$ receives a label in $\supp(v)$. For triangulations $\mathcal T_m$ of $\Delta_D$ whose mesh tends to zero, Bapat's lemma gives an $(L-1)$-dimensional simplex of $\mathcal T_m$ with vertices $v^{(m)}_x$, $x\in D$, and a bijection $x\mapsto c^{(m)}_x$ from $D$ onto $[L]$ such that $\lambda_{c^{(m)}_x}(v^{(m)}_x)=x$. After passing to a subsequence, the bijection does not depend on $m$, and by compactness all vertices $v^{(m)}_x$ converge to a common point $w^*\in\Delta_D$. From here on, the argument above applies with $v^{(m)}=v^{(m)}_x$. In this form, the proof uses only Sperner-type combinatorics and compactness.
\end{remark}

\section{Sharpness and limitations}
\label{sec:sharpness}

\subsection{One additional edge}

\begin{proposition}
\label{prop:sharp}
For every $k\ge1$ there is a path $T$ with singly coloured edges and colours $0,1,\ldots,k-1$ such that
\[
 |E_0|=k,\qquad |E_c|=k-1\quad(1\le c\le k-1),
\]
and no choice of one component of $T-E_c$ for every colour $c$ covers $V(T)$. Equivalently, there are connected partitions $\mathcal P_0,\ldots,\mathcal P_{k-1}$ of the vertex set of a path with $|\mathcal P_0|=k+1$ and $|\mathcal P_c|=k$ for $c\ge1$ such that no choice of one part from each partition covers all vertices.
\end{proposition}

\begin{proof}
For $k=1$, take a single edge of colour $0$. Let $k\ge2$, and let $T=v_0v_1\cdots v_N$ be the path with $N=k(k-1)+1$ edges in which the colours of the edges $v_0v_1,v_1v_2,\ldots$ form the word
\begin{equation}
 (0,1,\ldots,k-1)^{k-1}\,0,
 \label{eq:sharp-word}
\end{equation}
the concatenation of $k-1$ copies of $0,1,\ldots,k-1$ followed by a single $0$; see Figure~\ref{fig:sharp} for $k=3$. Numbering the edges $1,\ldots,N$ from the left, colour $0$ occurs at the positions $jk+1$ for $0\le j\le k-1$, and colour $c\ge1$ at the positions $jk+c+1$ for $0\le j\le k-2$. The components of $T-E_c$ are the maximal runs of edges avoiding colour $c$. Exactly $k-1$ edges separate consecutive occurrences of a colour. The outer components contain $0$ edges for colour $0$, and $c$ and $k-c$ edges for colour $c\ge1$. Hence, every component of every $T-E_c$ contains at most $k-1$ edges.

Suppose that components $K_c$ of $T-E_c$ cover $V(T)$. Then every edge $uv$ is in some $T[K_d]$. Indeed, let $c$ be the colour of $uv$. Since $uv$ is a bridge of $T$ and belongs to $E_c$, the vertices $u$ and $v$ lie in different components of $T-E_c$, so some endpoint of $uv$ lies outside $K_c$ and hence in some $K_d$ with $d\ne c$. As $uv$ is an edge of $T-E_d$, both endpoints lie in $K_d$. Therefore
\[
 k(k-1)+1=|E(T)|\le\sum_{c=0}^{k-1}|E(T[K_c])|\le k(k-1),
\]
a contradiction. The equivalent formulation follows from Lemma~\ref{lem:tree}(iv), applied to the partitions \mbox{$\mathcal P_c=\Comp(T-E_c)$}.
\end{proof}

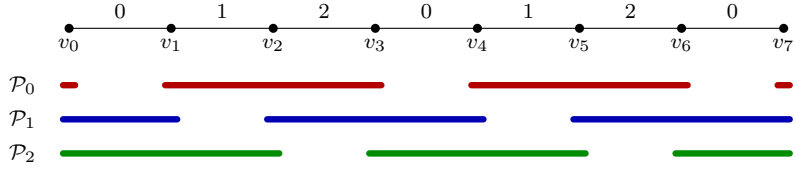
\begin{figure}[t]
\centering
\begin{tikzpicture}[scale=1.0,vtx/.style={circle,fill=black,inner sep=1.2pt}]
\foreach \i in {0,...,7} {\coordinate (v\i) at (1.35*\i,0);}
\foreach \i/\c [evaluate=\i as \j using int(\i+1)] in {0/0,1/1,2/2,3/0,4/1,5/2,6/0}
  {\draw (v\i)--(v\j) node[midway,above,font=\scriptsize]{$\c$};}
\foreach \i in {0,...,7} {\node[vtx] at (v\i) {}; \node[below,font=\scriptsize] at (v\i) {$v_{\i}$};}
\begin{scope}[yshift=-0.75cm]
\node[anchor=east,font=\scriptsize] at (-0.3,0) {$\mathcal P_0$};
\foreach \a/\b in {0/0,1/3,4/6,7/7} {\draw[line width=2.4pt,line cap=round,red!70!black] ($(1.35*\a,0)+(-0.08,0)$)--($(1.35*\b,0)+(0.08,0)$);}
\end{scope}
\begin{scope}[yshift=-1.2cm]
\node[anchor=east,font=\scriptsize] at (-0.3,0) {$\mathcal P_1$};
\foreach \a/\b in {0/1,2/4,5/7} {\draw[line width=2.4pt,line cap=round,blue!70!black] ($(1.35*\a,0)+(-0.08,0)$)--($(1.35*\b,0)+(0.08,0)$);}
\end{scope}
\begin{scope}[yshift=-1.65cm]
\node[anchor=east,font=\scriptsize] at (-0.3,0) {$\mathcal P_2$};
\foreach \a/\b in {0/2,3/5,6/7} {\draw[line width=2.4pt,line cap=round,green!55!black] ($(1.35*\a,0)+(-0.08,0)$)--($(1.35*\b,0)+(0.08,0)$);}
\end{scope}
\end{tikzpicture}
\caption{The path of Proposition~\ref{prop:sharp} for $k=3$, with edge colours $0120120$, and the partitions $\mathcal P_c=\Comp(T-E_c)$ drawn as intervals. Every part contains at most two of the seven edges, so three chosen parts contain at most six edges. A covering would have to contain every edge inside some chosen part (see the proof), so no three parts cover the path.}
\label{fig:sharp}
\end{figure}

\subsection{Paths and permutation words}
Let $T=v_0v_1\cdots v_N$ be a path with singly coloured edges, the edge $v_{p-1}v_p$ having colour $\sigma_p\in C$, where $|C|=k$. Thus $E_c=\{v_{p-1}v_p:\sigma_p=c\}$. We call $\sigma=\sigma_1\cdots\sigma_N$ the \emph{colour word} of $T$. An ordering $(c_1,\ldots,c_k)$ of $C$ is a \emph{subsequence} of $\sigma$ if there are positions $p_1<\cdots<p_k$ with $\sigma_{p_j}=c_j$ for all $j$.

\begin{proposition}
\label{prop:paths}
For a path $T$ with singly coloured edges and colour word $\sigma$, the following are equivalent.
\begin{enumerate}
\item The instance $(T,(E_c)_{c\in C})$ has a solution.
\item One can choose a component $K_c$ of $T-E_c$ for every $c\in C$ so that $\bigcup_cK_c=V(T)$.
\item Some ordering of $C$ is not a subsequence of $\sigma$.
\end{enumerate}
\end{proposition}

\begin{proof}
(i)$\Rightarrow$(ii) is Proposition~\ref{prop:equivalence}.

(ii)$\Rightarrow$(i). The sets $K_c$ are intervals of the path. Construct the solution from left to right. Let $x=0$ and let $U=C$ be the set of unused colours. Among the colours $c\in U$ with $v_x\in K_c$, choose one whose interval $K_c$ extends furthest to the right, say to $v_y$; put $B_c=\{v_x,\ldots,v_y\}\subseteq K_c$, remove $c$ from $U$, and continue with $x=y+1$ as long as $x\le N$. Every colour chosen earlier had an interval ending before $v_x$. Since the sets $K_c$ cover $V(T)$, some colour has an interval containing $v_x$, and every such colour is unused, so the procedure succeeds. Unused colours receive empty sets. Each $B_c$ is a subset of $K_c$ and therefore contains no edge of colour $c$.

(i)$\Rightarrow$(iii). Let $(B_c)$ be a solution. Its nonempty sets are intervals $B_{c_1},\ldots,B_{c_q}$, listed from left to right; write $B_{c_j}=\{v_{b_{j-1}+1},\ldots,v_{b_j}\}$ with $b_0=-1$ and $b_q=N$. Extend $(c_1,\ldots,c_q)$ to an ordering $(c_1,\ldots,c_k)$ of $C$, and suppose that $p_1<\cdots<p_k$ are positions with $\sigma_{p_j}=c_j$. The edges inside $B_{c_j}$ are the positions $b_{j-1}+2,\ldots,b_j$, none of which has colour $c_j$. By induction on $j$, we get $p_j\ge b_j+1$: for $j=1$ this holds because $p_1\ge1$ and $p_1$ is not inside $B_{c_1}$, and for $j\ge2$ because $p_j>p_{j-1}\ge b_{j-1}+1$ and $p_j$ is not inside $B_{c_j}$. For $j=q$ this gives $p_q\ge N+1$, which is impossible.

(iii)$\Rightarrow$(i). Let $(c_1,\ldots,c_k)$ be an ordering that is not a subsequence of $\sigma$. Put $t_0=0$, and for $j\ge1$ let $t_j$ be the first position after $t_{j-1}$ carrying colour $c_j$, as long as such a position exists. Since the ordering is not a subsequence, there is a first index $j_0\le k$ for which $t_{j_0}$ does not exist. Put $B_{c_j}=\{v_{t_{j-1}},\ldots,v_{t_j-1}\}$ for $j<j_0$, put $B_{c_{j_0}}=\{v_{t_{j_0-1}},\ldots,v_N\}$, and let the remaining colours receive empty sets. By the choice of the positions $t_j$, no set $B_{c_j}$ contains an edge of colour $c_j$.
\end{proof}

\begin{corollary}[Paths]
\label{cor:paths}
If every letter occurs at most $k-1$ times in a word over a $k$-letter alphabet, then some ordering of the alphabet is not a subsequence of the word. Consequently, Theorem~\ref{thm:Aprime} holds for paths.
\end{corollary}

\begin{proof}
For $k=1$ the word is empty, and the claim is trivial, so let $k\ge2$. Suppose that all orderings are subsequences; then every letter occurs. Let $c_1$ be the letter whose first occurrence is the latest, and let $p_1$ be this occurrence. Recursively, for $j=2,\ldots,k$, every letter $c\notin\{c_1,\ldots,c_{j-1}\}$ occurs after $p_{j-1}$, because an ordering beginning with $c_1,\ldots,c_{j-1},c$ is a subsequence, and if $q_1<\cdots<q_{j-1}$ are positions of $c_1,\ldots,c_{j-1}$, then $q_i\ge p_i$ for every $i$, by induction on $i$, since $p_1$ is the first occurrence of $c_1$ and $p_i$ is the first occurrence of $c_i$ after $p_{i-1}$. Let $c_j$ be the letter among them whose first occurrence after $p_{j-1}$ is the latest, and let $p_j$ be this occurrence. The letter $c_k$ then occurs before $p_1$, between $p_{j-1}$ and $p_j$ for every $2\le j\le k-1$, and at $p_k$, which gives at least $k$ occurrences. For the second statement, apply Lemma~\ref{lem:normalization}(i) to an admissible instance on a path. Contracting the uncoloured edges, whose components are subpaths, and subdividing the edges with several colours turn a path into a path; we obtain an admissible instance on a path with singly coloured edges and $|\widehat E_c|=|E_c|\le k-1$ for every $c$, a solution of which yields a solution of the original instance. By the first statement, some ordering of $C$ is not a subsequence of its colour word, so it has a solution by Proposition~\ref{prop:paths}. The covering statement then follows from Proposition~\ref{prop:equivalence}.
\end{proof}

Proposition~\ref{prop:paths} also shows that the individual bounds in Theorem~\ref{thm:Aprime} cannot be replaced by a bound on their sum, which would be natural in view of the total count $k(k-1)$ in a saturated instance.

\begin{proposition}
\label{prop:average}
For every $k\ge4$ there is a path $T$ with singly coloured edges and $k$ colours such that $\sum_c|E_c|=k^2-2k+4\le k(k-1)$, with strict inequality for $k\ge5$, and no choice of one component of $T-E_c$ for every colour covers $V(T)$.
\end{proposition}

\begin{proof}
For every $k\ge3$ there are words of length $k^2-2k+4$ over a $k$-letter alphabet containing every ordering of the alphabet as a subsequence~\cite{Newey1973,Adleman1974,KoutasHu1975}; see~\cite{EngenVatter2021}. Use such a word as the colour word of a path and apply Proposition~\ref{prop:paths}.
\end{proof}

For example, for $k=5$ the shortest words over a five-letter alphabet containing all $120$ orderings as subsequences have length $19=k^2-2k+4<k(k-1)$~\cite{Newey1973}; see~\cite[Section~4]{EngenVatter2021}. By Corollary~\ref{cor:paths}, some letter of such a word occurs at least $5$ times.

\subsection{Cycles}
The four-cycle example of Section~\ref{sec:why} generalises as follows.

\begin{proposition}[Cycles]
\label{prop:cycles}
Let $G$ be a cycle and $k\ge2$.
\begin{enumerate}
\item If $\mathcal P_1,\ldots,\mathcal P_k$ are connected partitions of $V(G)$ with $|\mathcal P_i|\le k-1$, then there is a rainbow partition subordinate to $(\mathcal P_i)$ in $G$, that is, an allocation $(B_i)_{i\in[k]}$ of $V(G)$ in which every nonempty $B_i$ is connected in $G$ and contained in a member of $\mathcal P_i$.
\item If $G$ has at least four vertices, then there are connected partitions $\mathcal P_1,\mathcal P_2$ of $V(G)$ with two parts each such that no part of $\mathcal P_1$ together with a part of $\mathcal P_2$ covers $V(G)$.
\item If $G$ has exactly $k^2$ vertices, then there are connected partitions $\mathcal P_1,\ldots,\mathcal P_k$ of $V(G)$ with $|\mathcal P_i|=k$ that admit no rainbow partition.
\end{enumerate}
\end{proposition}

\begin{proof}
(i) Delete an edge $uv$ of $G$ to obtain a path $T$. A member $X$ of $\mathcal P_i$ not containing both $u$ and $v$ satisfies $T[X]=G[X]$. If $X\ne V(G)$ contains both $u$ and $v$, then $G[X]$ is a path containing the edge $uv$, and $X$ splits into two connected sets of $T$; if $X=V(G)$, then $X$ is connected in $T$. Since at most one member of $\mathcal P_i$ contains both $u$ and $v$, we obtain connected partitions of $T$ with at most $k$ parts that refine the partitions $\mathcal P_i$. Theorem~\ref{thm:A} for $T$ gives the claim, because connected sets of $T$ are connected in $G$.

(ii) Let $e_1,f_1,e_2,f_2$ be distinct edges of $G$ in this cyclic order, and let $\mathcal P_1=\Comp(G-\{e_1,e_2\})$ and $\mathcal P_2=\Comp(G-\{f_1,f_2\})$. Each part of $\mathcal P_1$ and each part of $\mathcal P_2$ is the union of two of the four components of $G-\{e_1,f_1,e_2,f_2\}$, and a part of $\mathcal P_1$ and a part of $\mathcal P_2$ always share one of these components. Hence, their union misses one of the four components.

(iii) Let $G=v_0v_1\cdots v_{k^2-1}v_0$, with indices taken modulo $k^2$, and for $i\in[k]$ let $\mathcal P_i$ consist of the $k$ sets $\{v_{i+tk},v_{i+tk+1},\ldots,v_{i+tk+k-1}\}$, $0\le t\le k-1$. These are connected partitions with $|\mathcal P_i|=k$, and all their members have exactly $k$ vertices. If $(B_i)_{i\in[k]}$ were a rainbow partition subordinate to them, then $|B_i|\le k$ and $\sum_i|B_i|=k^2$, so every $B_i$ would be a member of $\mathcal P_i$. Thus $V(G)$ would be partitioned into $k$ sets of $k$ consecutive vertices. Two such sets that follow each other along the cycle have first vertices whose indices differ by $k$, so all first indices would be congruent modulo $k$; this is well defined because $k$ divides $k^2$. But the first index of $B_i$ is congruent to $i$ modulo $k$, and $1\not\equiv2\pmod k$, a contradiction.
\end{proof}

For $k=2$, part~(iii) is the four-cycle of Figure~\ref{fig:small}(a). Thus on cycles the bound $k-1$ in~(i) is best possible for every $k\ge2$, whereas the triangle satisfies the conclusion of Theorem~\ref{thm:A} for every $k$: every nonempty set of vertices of a triangle is connected, so for $k\ge3$ singletons suffice, and for $k=2$ one can take a largest part of $\mathcal P_1$ as $B_1$ and its complement, which has at most one vertex, as $B_2$.

\subsection{Covers versus partitions}
\label{sec:covers}
A covering by components need not contain a solution, even under the hypothesis of Theorem~\ref{thm:Aprime}. Let $T$ be the tree with edges $xa,xy,xc,xd,yz$, let $k=3$, and let $E_1=\{xy,yz\}$, $E_2=\{xa,xy\}$, and $E_3=\{xc,xd\}$. The components $K_1=\{y\}$ of $T-E_1$, $K_2=\{x,c,d\}$ of $T-E_2$, and $K_3=\{x,a,y,z\}$ of $T-E_3$ cover $V(T)$. Suppose that a solution satisfies $B_i\subseteq K_i$ for all $i$. The vertices $c$ and $d$ lie in no $K_i$ other than $K_2$, and $a$ and $z$ lie in no $K_i$ other than $K_3$. Hence $B_2\supseteq\{c,d\}$ and $B_3\supseteq\{a,z\}$, and both connected sets contain $x$, a contradiction. (The solution $B_1=\{x,a,c,d\}$, $B_2=\{y,z\}$, $B_3=\varnothing$ uses other components.) On paths, by contrast, every covering by components contains a solution, by the proof of Proposition~\ref{prop:paths}. Thus, the partition statement of Theorem~\ref{thm:Aprime} cannot be obtained by trimming an arbitrary covering.

\section{Fair division consequences}
\label{sec:chores}

We use the model of Section~\ref{sec:results}: the vertices of a tree $T$ are chores, the $n$ agents have monotone costs $\kappa_i$, and $\MMS_i(T,n)$ is defined by~\eqref{eq:intro-mms} as a minimum over the finite nonempty family $\Pi_n(T)$, which contains $\{V(T)\}$.

\subsection{Conventions}

\begin{lemma}[Exactly $n$ bundles]
\label{lem:exactly}
If $|V(T)|\ge n$, then
\[
 \MMS_i(T,n)=\min_{\mathcal P}\ \max_{P\in\mathcal P}\ \kappa_i(P),
\]
where $\mathcal P$ ranges over the connected partitions of $V(T)$ with exactly $n$ parts. If $|V(T)|<n$, then there is no such partition. For every $T$, the benchmark $\MMS_i(T,n)$ also equals the minimum of $\max_{j\in[n]}\kappa_i(B_j)$ over all connected allocations $(B_j)_{j\in[n]}$ of $V(T)$.
\end{lemma}

\begin{proof}
One inequality is immediate. For the other, let $\mathcal P\in\Pi_n(T)$ have $q<n$ parts. Since $|V(T)|\ge n>q$, some part $P$ has at least two vertices, and deleting an edge of the tree $T[P]$ splits $P$ into two connected sets, each of cost at most $\kappa_i(P)$ by monotonicity. After $n-q$ such steps, we reach a partition with exactly $n$ parts and no higher maximum cost. For the last statement, the nonempty members of a connected allocation form a partition in $\Pi_n(T)$ with the same maximum cost, because empty bundles cost $0$ and some bundle is nonempty; conversely, every member of $\Pi_n(T)$ becomes a connected allocation after adding empty bundles.
\end{proof}

\begin{remark}[The convention of Xiao, Qiu, and Huang]
\label{rem:xqh}
In~\cite{XiaoEtAl2023arXiv}, agent $i$ has an additive disutility $u_i\le0$, a \emph{valid $n$-partition} is an assignment of $n$ bundles inducing connected subgraphs, and $\mms_i$ is the maximum over valid $n$-partitions of the minimum utility of a bundle. With $u_i=-\kappa_i$, Lemma~\ref{lem:exactly} gives $\mms_i=-\MMS_i(T,n)$ whenever $|V(T)|\ge n$, whether or not empty bundles are admitted in valid partitions. Hence, when $|V(T)|\ge n$, their MMS allocations are exactly our connected MMS allocations, with all bundles nonempty if their definition requires it, and Theorem~\ref{thm:B}(a) provides such allocations.
\end{remark}

\subsection{Proof of Theorem~\ref{thm:B}}

\begin{proof}[Proof of Theorem~\ref{thm:B}(a)]
For every $i$, choose $\mathcal P_i\in\Pi_n(T)$ attaining $\MMS_i(T,n)$; if such partitions are prescribed, use them. Theorem~\ref{thm:A} with $k=n$ gives a connected allocation $(B_i)$ in which every nonempty $B_i$ lies in some $P_i\in\mathcal P_i$, so that
\[
 \kappa_i(B_i)\le\kappa_i(P_i)\le\MMS_i(T,n)
\]
by monotonicity; for $B_i=\varnothing$ the inequality holds because $\kappa_i(\varnothing)=0$. If $|V(T)|\ge n$, the pieces can be chosen nonempty by Theorem~\ref{thm:A}.
\end{proof}

The converse rests on the following construction, in which the costs are chosen so that the benchmark partition is essentially the only way to satisfy the agent.

\begin{lemma}
\label{lem:costs}
Let $T$ be a tree, let $k\ge1$, let $\mathcal P$ be a connected partition of $V(T)$ with exactly $k$ parts, and let $\eta=1/(4k)$. There is an additive cost $\kappa$, with vertex costs $\kappa(v)\ge0$, such that $\kappa(P)\le1-\eta$ for every $P\in\mathcal P$, and every connected set $S$ with $\kappa(S)\le1-\eta$ is contained in a member of $\mathcal P$.
\end{lemma}

\begin{proof}
Root the quotient tree $T/\mathcal P$ at a part $R$. For a part $X\ne R$, let $u_Xv_X$ be the unique edge of $T$ joining $X$ to its parent part, with $v_X\in X$, and let $m_X\in[k-1]$ be the number of parts in the subtree of $T/\mathcal P$ rooted at $X$. Starting from $\kappa\equiv0$, add, for every part $X\ne R$, the amount $1-\eta(2m_X-1)$ to $\kappa(v_X)$ and the amount $2\eta\,m_X$ to $\kappa(u_X)$. Since $2m_X-1\le 2k-3$, all amounts are nonnegative.

A part $X\ne R$ receives the amount $1-\eta(2m_X-1)$ at $v_X$ and the amounts $2\eta\,m_Y$ at $u_Y\in X$ for its children $Y$, and nothing else. Since $\sum_Ym_Y=m_X-1$, we get
\[
 \kappa(X)=1-\eta(2m_X-1)+2\eta(m_X-1)=1-\eta.
\]
 The root receives $\kappa(R)=2\eta(k-1)<1/2\le1-\eta$. Every edge of $T$ joining two distinct parts is of the form $u_Xv_X$, and
\[
 \kappa(u_X)+\kappa(v_X)\ge2\eta\,m_X+1-\eta(2m_X-1)=1+\eta .
\]
A connected set $S$ meeting two parts contains two adjacent vertices in different parts, and hence both endpoints of such an edge; thus $\kappa(S)\ge1+\eta$.
\end{proof}

\begin{proof}[Proof of Theorem~\ref{thm:B}(b)]
Let $\mathcal P_1,\ldots,\mathcal P_n$ be connected partitions of $V(T)$ with $|\mathcal P_i|\le n$. If $|V(T)|\le n$, assigning the vertices to distinct indices gives a rainbow partition. Otherwise, refine every $\mathcal P_i$ to a connected partition $\mathcal P_i'$ with exactly $n$ parts by splitting parts along edges, as in the proof of Lemma~\ref{lem:exactly}. Let $\kappa_i$ be the cost given by Lemma~\ref{lem:costs} for $\mathcal P'_i$ and $k=n$. Then $\MMS_i(T,n)\le1-\eta$. By hypothesis, there is a connected MMS allocation $(B_i)$ for these costs. Every nonempty $B_i$ is connected with $\kappa_i(B_i)\le1-\eta$, so it lies in a member of $\mathcal P'_i$ and hence in a member of $\mathcal P_i$. By Proposition~\ref{prop:nonempty}, the pieces can moreover be chosen nonempty.
\end{proof}

\subsection{The packing counterpart for goods}
The following proposition is a combinatorial form of the last-diminisher argument of Bouveret et~al.~\cite{BouveretEtAl2017}. It explains why the case of goods on trees is easy.

\begin{proposition}
\label{prop:goods}
Let $T$ be a tree, let $k\ge1$, and let $\mathcal F_1,\ldots,\mathcal F_k$ be families, each consisting of at least $k$ pairwise disjoint connected subsets of $V(T)$. Then there is a connected partition $\{A_1,\ldots,A_k\}$ of $V(T)$, indexed by $[k]$, such that every $A_i$ contains a member of $\mathcal F_i$.
\end{proposition}

\begin{proof}
We use induction on $k$; for $k=1$ take $A_1=V(T)$. Let $k\ge2$. Root $T$ at a vertex $\rho$, and for a vertex $v$ let $T_v$ be the set consisting of $v$ and its descendants. Choose a vertex $v$ of maximum depth such that $T_v$ contains a member of some family, say of $\mathcal F_i$. If $X$ is a member of any family that meets $T_v$ but does not contain $v$, then $X$ lies in a component of $T-v$ meeting $T_v$, that is, in $T_u$ for a child $u$ of $v$, contrary to the choice of $v$. Hence, every member of any family meeting $T_v$ contains $v$. In particular, since $\mathcal F_i$ has at least two disjoint members, $v\ne\rho$, and $T-T_v$ is a tree. Because the members of each family are pairwise disjoint, every family $\mathcal F_j$ with $j\ne i$ has at least $k-1$ members contained in $V(T)\setminus T_v$, which are connected in $T-T_v$. The induction hypothesis applied to these subfamilies on $T-T_v$ yields a connected partition $\{A_j:j\ne i\}$ of $V(T)\setminus T_v$, and we put $A_i=T_v$.
\end{proof}

For goods with monotone valuations, the maximin share of an agent is the maximum, over connected partitions of $V(T)$ into exactly $n$ parts, of the minimum value of a part (it is $0$ if $|V(T)|<n$). Applying Proposition~\ref{prop:goods} to the parts of partitions attaining the agents' maximin shares yields connected MMS allocations of goods on trees, which recovers the existence part of~\cite{BouveretEtAl2017}. Theorems~\ref{thm:A} and~\ref{thm:B} are the covering counterparts of Proposition~\ref{prop:goods}.

\section{Beyond trees}
\label{sec:beyond}

A very natural question is whether the conclusions of Theorem~\ref{thm:A} and Proposition~\ref{prop:goods} extend to graphs that are not trees. We show that the classification of such graphs reduces to the case of trees, and hence, by our results, is complete.

Throughout this section, $G$ is a finite connected graph. The notions of a connected set, a connected partition, and an allocation are as in Section~\ref{sec:results}, with $G$ in place of $T$. We say that $G$ has the \emph{rainbow partition property} if, for every $k\ge1$ and all connected partitions $\mathcal P_1,\ldots,\mathcal P_k$ of $V(G)$ with $|\mathcal P_i|\le k$, there is an allocation $(B_i)_{i\in[k]}$ of $V(G)$ in which every nonempty $B_i$ is connected and contained in a member of $\mathcal P_i$. We say that $G$ has the \emph{packing property} if, for every $k$ with $1\le k\le|V(G)|$ and all connected partitions $\mathcal P_1,\ldots,\mathcal P_k$ of $V(G)$ with exactly $k$ parts each, there is a connected partition $\{A_1,\ldots,A_k\}$ of $V(G)$, indexed by $[k]$, such that every $A_i$ contains a member of $\mathcal P_i$. By Theorem~\ref{thm:A} and Proposition~\ref{prop:goods}, every tree has both properties. As in the proofs of Theorem~\ref{thm:B}(a) and of the goods counterpart after Proposition~\ref{prop:goods}, the rainbow partition property yields connected MMS allocations of chores for monotone costs, and the packing property yields connected MMS allocations of goods for monotone valuations, where the benchmarks are defined by connected partitions of $V(G)$ exactly as for trees.

\subsection{Graphs with a long cycle}
The following lemma extends Proposition~\ref{prop:cycles}(ii) from cycles to all graphs containing a cycle of length at least four, and it treats goods as well.

\begin{lemma}
\label{lem:long-cycle}
If $G$ contains a cycle of length at least four as a subgraph, then there are connected partitions $\mathcal P_1,\mathcal P_2$ of $V(G)$ with two parts each for which there is neither a rainbow partition subordinate to $(\mathcal P_1,\mathcal P_2)$ nor a connected partition $\{A_1,A_2\}$ of $V(G)$ with $A_i$ containing a member of $\mathcal P_i$ for $i=1,2$. In particular, $G$ has neither the rainbow partition property nor the packing property.
\end{lemma}

\begin{proof}
Let $C$ be a cycle of length $\ell\ge4$ in $G$, denote its consecutive vertices by $1,2,\ldots,\ell$, and put $m=\lfloor\ell/2\rfloor$, so that $2\le m\le\ell-2$. Define
\begin{align*}
 A_1&=\{1,\ldots,m\}, & A_2&=\{m+1,\ldots,\ell\},\\
 B_1&=\{2,\ldots,m+1\}, & B_2&=\{m+2,\ldots,\ell,1\}.
\end{align*}
Each of these four sets spans a path in $C$, so it is connected in $G$. They have the following two properties.
\begin{enumerate}[label=\textup{(\alph*)}]
\item Every set $A_j$ meets every set $B_{j'}$: indeed $2\in A_1\cap B_1$, $1\in A_1\cap B_2$, $m+1\in A_2\cap B_1$, and $\ell\in A_2\cap B_2$.
\item No set $A_j$ together with a set $B_{j'}$ covers $V(C)$: the unions $A_1\cup B_1$, $A_1\cup B_2$, $A_2\cup B_1$, and $A_2\cup B_2$ miss the vertices $\ell$, $m+1$, $1$, and $2$, respectively.
\end{enumerate}
Since $G$ is connected, we can extend $\{A_1,A_2\}$ and $\{B_1,B_2\}$ to connected partitions $\mathcal P_1=\{A_1',A_2'\}$ and $\mathcal P_2=\{B_1',B_2'\}$ of $V(G)$ by adding the vertices of $V(G)\setminus V(C)$ one at a time, each to a current part containing one of its neighbours. Then $A_j'\cap V(C)=A_j$ and $B_{j'}'\cap V(C)=B_{j'}$ for all $j,j'$.

Suppose that $(B_1^\circ,B_2^\circ)$ is a rainbow partition subordinate to $(\mathcal P_1,\mathcal P_2)$. Then $B_1^\circ\cap V(C)\subseteq A_j$ and $B_2^\circ\cap V(C)\subseteq B_{j'}$ for some $j,j'$, and by~(b) the union $B_1^\circ\cup B_2^\circ$ misses a vertex of $C$, a contradiction. Suppose next that $\{A_1^\circ,A_2^\circ\}$ is a connected partition with $A_1^\circ\supseteq A_j'$ and $A_2^\circ\supseteq B_{j'}'$ for some $j,j'$. Then $A_1^\circ\cap A_2^\circ\supseteq A_j\cap B_{j'}\ne\varnothing$ by~(a), again a contradiction.
\end{proof}

\subsection{Triangular cacti}
Lemma~\ref{lem:long-cycle} leaves exactly the connected graphs with no cycle of length at least four. A \emph{triangular cactus} is a connected graph in which every block (maximal connected subgraph without a cut vertex) is a single edge or a triangle; in particular, every tree is a triangular cactus, and in a triangular cactus, distinct triangles share no edge.

\begin{lemma}
\label{lem:no-long-cycle}
A connected graph $G$ contains no cycle of length at least four if and only if it is a triangular cactus.
\end{lemma}

\begin{proof}
Every cycle of a graph lies in one of its blocks, so a triangular cactus has only triangles as cycles. Conversely, suppose that $G$ has a block $H$ with at least four vertices. Then $H$ is $2$-connected. Let $C$ be a longest cycle of $H$; if $|C|\ge4$ we are done, so assume that $C=xyz$ is a triangle, and let $v$ be a vertex of $H$ outside $C$. By Menger's theorem, $H$ contains two paths from $v$ to $C$ that share only $v$ and end at distinct vertices of $C$, say $x$ and $y$. Together with the path $x\,z\,y$, they form a cycle of length at least four, a contradiction.
\end{proof}

\begin{lemma}
\label{lem:cacti}
Every triangular cactus has the rainbow partition property and the packing property.
\end{lemma}

\begin{proof}
Let $G$ be a triangular cactus, and let $q$ be the number of its triangles. Let $T$ be the graph obtained from $G$ by adding, for every triangle $xyz$, a new vertex $t_{xyz}$ adjacent to $x$, $y$, and $z$, and deleting the edges $xy$, $xz$, and $yz$. Every deleted edge $xy$ is replaced by the path $x\,t_{xyz}\,y$, so $T$ is connected. Since every cycle of $G$ lies in a block, the cycles of $G$ are exactly its $q$ triangles, and each block that is a triangle has one edge more than a tree on its vertices; hence $|E(G)|=|V(G)|-1+q$. Hence $|E(T)|=|E(G)|=|V(T)|-1$, and $T$ is a tree.

Let $\mathcal P$ be a connected partition of $V(G)$. We extend it to a partition $\mathcal P'$ of $V(T)$ as follows. For every triangle $xyz$, at most one member of $\mathcal P$ contains two or more of the vertices $x,y,z$, since the members are disjoint. If such a member exists, we add $t_{xyz}$ to it. Otherwise, we add $t_{xyz}$ to an arbitrary member containing one of $x,y,z$. Every member $P$ of $\mathcal P$ thus gives a member $P'$ of $\mathcal P'$ with $P'\cap V(G)=P$, and $|\mathcal P'|=|\mathcal P|$. Each $P'$ is connected in $T$: every new vertex of $P'$ is adjacent in $T$ to a vertex of $P$, every edge of $G[P]$ not belonging to a triangle is an edge of $T$, and if an edge $xy$ of a triangle $xyz$ lies in $G[P]$, then $P$ contains two vertices of the triangle, so $t_{xyz}\in P'$ and the path $x\,t_{xyz}\,y$ lies in $T[P']$.

Conversely, if $R'\subseteq V(T)$ is connected in $T$, then $R=R'\cap V(G)$ is empty or connected in $G$. Indeed, a path of $T[R']$ between two vertices of $R$ visits a new vertex $t_{xyz}$ only between two of its neighbours $x,y,z$, which are adjacent in $G$; replacing every such detour by the corresponding edge of $G$ gives a walk in $G[R]$.

Now let $\mathcal P_1,\ldots,\mathcal P_k$ be connected partitions of $V(G)$ with $|\mathcal P_i|\le k$, and let $\mathcal P'_1,\ldots,\mathcal P'_k$ be the extended partitions of $V(T)$, which satisfy $|\mathcal P'_i|\le k$. By Theorem~\ref{thm:A} there is a rainbow partition $(R'_i)_{i\in[k]}$ of $V(T)$ subordinate to $(\mathcal P'_i)$. Put $R_i=R'_i\cap V(G)$. These sets form an allocation of $V(G)$, every nonempty $R_i$ is connected in $G$, and if $R'_i\subseteq P'$ with $P'\in\mathcal P'_i$, then $R_i\subseteq P'\cap V(G)$, which is a member of $\mathcal P_i$. Hence $G$ has the rainbow partition property.

For the packing property, let $k\le|V(G)|$ and let each $\mathcal P_i$ have exactly $k$ parts. The extended partitions $\mathcal P'_i$ of $V(T)$ have exactly $k$ parts, which are pairwise disjoint connected sets of $T$. Proposition~\ref{prop:goods}, applied to the families $\mathcal P'_1,\ldots,\mathcal P'_k$, gives a connected partition $\{A'_1,\ldots,A'_k\}$ of $V(T)$ with $A'_i\supseteq P'$ for some $P'\in\mathcal P'_i$. Then $A_i=A'_i\cap V(G)$ contains the member $P'\cap V(G)$ of $\mathcal P_i$, so it is nonempty and, by the previous paragraph, connected in $G$. The sets $A_i$ form the required connected partition of $V(G)$.
\end{proof}

\subsection{The characterization}
Combining Lemmas~\ref{lem:long-cycle}, \ref{lem:no-long-cycle}, and~\ref{lem:cacti}, we obtain a complete answer.

\begin{theorem}
\label{thm:cacti}
For a connected graph $G$, the following statements are equivalent.
\begin{enumerate}
\item $G$ is a triangular cactus.
\item $G$ contains no cycle of length at least four.
\item $G$ has the rainbow partition property.
\item $G$ has the packing property.
\end{enumerate}
\end{theorem}

\begin{proof}
Statements (i) and (ii) are equivalent by Lemma~\ref{lem:no-long-cycle}. Statement (i) implies (iii) and (iv) by Lemma~\ref{lem:cacti}, and each of (iii) and (iv) implies (ii) by Lemma~\ref{lem:long-cycle}.
\end{proof}

In terms of fair division, let $G$ be a connected graph of chores and let $\MMS_i(G,n)$ be defined by~\eqref{eq:intro-mms} with connected partitions of $V(G)$ into at most $n$ parts. If $G$ is a triangular cactus, then connected MMS allocations of chores exist for all monotone costs, by Theorem~\ref{thm:cacti} and the proof of Theorem~\ref{thm:B}(a). If $G$ is not a triangular cactus, take the partitions $\mathcal P_1,\mathcal P_2$ of Lemma~\ref{lem:long-cycle} and, for $n=2$, the monotone costs $\kappa_i(X)=0$ if $X$ is contained in a member of $\mathcal P_i$ and $\kappa_i(X)=1$ otherwise. Then $\MMS_i(G,2)=0$, and a connected MMS allocation would be a rainbow partition subordinate to $(\mathcal P_1,\mathcal P_2)$, which does not exist. Thus, a connected graph admits connected MMS allocations of chores for all monotone costs and all numbers of agents if and only if it is a triangular cactus; the analogous statement for goods follows in the same way from the packing property.

\section{Algorithms}
\label{sec:algorithm}

The proof of Theorem~\ref{thm:weights} is not constructive. Nevertheless, a rainbow partition can be found by a finite search whose superpolynomial part depends only on the number of colours. We use the standard RAM model and assume that the tree and the edge sets (or the partitions) are given explicitly.

\begin{theorem}
\label{thm:algorithm}
Given a tree $T$ and edge sets $E_1,\ldots,E_k$ with $|E_i|\le k-1$, a solution as in Theorem~\ref{thm:Aprime} can be computed in time $k^{O(k)}+\poly(|V(T)|,k)$. The same holds for a rainbow partition subordinate to connected partitions $\mathcal P_1,\ldots,\mathcal P_k$ with $|\mathcal P_i|\le k$.
\end{theorem}

\begin{proof}
For $k=1$ the problem is trivial, so let $k\ge2$. Apply Lemma~\ref{lem:normalization}(i). The resulting tree $\widehat T$ has singly coloured edges, and it has
\[
 N=|E(\widehat T)|=\sum_{c}|\widehat E_c|\le k(k-1)
\]
edges; thus, normalisation produces, in polynomial time, an equivalent instance of size bounded by a function of $k$ alone. If $N=0$, then $\widehat T$ has a single vertex, which forms a solution by itself, so let $N\ge1$. A connected partition of $V(\widehat T)$ into $q$ parts is determined by the set $F$ of $q-1$ edges between its parts (Lemma~\ref{lem:tree}(iii),(iv)). For every $F\subseteq E(\widehat T)$ with $|F|\le k-1$, form the bipartite graph between $\CC F=\Comp(\widehat T-F)$ and the colours in which $Z\in\CC F$ is adjacent to $c$ if $E(\widehat T[Z])\cap\widehat E_c=\varnothing$, and test whether it has a matching covering $\CC F$. Such a matching yields a solution, with empty sets for unmatched colours. Conversely, the nonempty sets of any solution, which exists by Theorem~\ref{thm:Aprime}, form a connected partition into $q\le k$ parts; its set $F$ of boundary edges has $q-1\le k-1$ elements, $\CC F$ consists of exactly these parts, and the solution defines a matching covering $\CC F$. Hence, the search succeeds. The number of candidate sets is
\[
 \sum_{j=0}^{k-1}\binom Nj\le k\,N^{k-1}\le k^{2k-1},
\]
and each is processed in time polynomial in $k$, for instance with the Hopcroft--Karp algorithm~\cite{HopcroftKarp1973}. Transferring the solution back to $T$ takes polynomial time by Lemma~\ref{lem:normalization}(i). For input partitions, first compute the edge sets $E_i$ of edges joining distinct members of $\mathcal P_i$, as in Proposition~\ref{prop:equivalence}.
\end{proof}

\begin{corollary}
\label{cor:mms-algorithm}
Given a tree $T$ and nonnegative rational additive costs $\kappa_1,\ldots,\kappa_n$ on $V(T)$, a connected MMS allocation, with nonempty bundles if $|V(T)|\ge n$, can be computed in time $n^{O(n)}$ plus a polynomial in the input length. In particular, it can be computed in polynomial time for every fixed number of agents.
\end{corollary}

\begin{proof}
If $|V(T)|\le n$, give the vertices to distinct agents. This is an MMS allocation: every partition in $\Pi_n(T)$ has a part containing a vertex of maximum cost for agent $i$, so $\MMS_i(T,n)$ is at least the largest vertex cost $\max_v\kappa_i(v)$. If $|V(T)|>n$, then by Lemma~\ref{lem:exactly} every optimal solution of the min--max tree partitioning problem with vertex weights $\kappa_i$ and exactly $n-1$ deleted edges is a partition in $\Pi_n(T)$ attaining $\MMS_i(T,n)$. Such a solution can be computed in polynomial time by the shifting algorithm of Becker, Schach, and Perl~\cite{BeckerEtAl1982}; see also~\cite{Frederickson1991}. Apply Theorem~\ref{thm:algorithm} to the $n$ partitions obtained, and then Proposition~\ref{prop:nonempty}; the proof of Theorem~\ref{thm:B}(a) shows that the result is an MMS allocation.
\end{proof}

\begin{remark}
\label{rem:tfnp}
A rainbow partition always exists and can be verified in polynomial time, so the search problem of finding one belongs to the class TFNP of total search problems in NP~\cite{MegiddoPapadimitriou1991}. The same holds for connected MMS allocations of additive chores on trees, since the benchmarks can be computed in polynomial time (see the proof of Corollary~\ref{cor:mms-algorithm}). We do not know whether these problems can be solved in time polynomial in $|V(T)|$ and $k$, or whether they belong to the class PPAD~\cite{Papadimitriou1994}, which contains the search problems associated with Sperner's lemma. The labellings in the proof of Theorem~\ref{thm:weights} are discontinuous, and the limit argument provides no bound on the fineness of an approximation that would suffice.
\end{remark}

\section{Open problems}
\label{sec:open}

\begin{problem}[Complexity]
Can a rainbow partition subordinate to $k$ connected partitions of a tree, or a connected MMS allocation of additive chores on a tree, be computed in time polynomial in the size of the tree and the number of agents? Is the search problem in PPAD, or hard for it? By Lemma~\ref{lem:normalization}(i), it suffices to treat singly coloured instances with at most $k(k-1)$ edges.
\end{problem} 

\begin{problem}[Graphs with long cycles]
By Theorem~\ref{thm:cacti}, the conclusion of Theorem~\ref{thm:A} holds for all $k$ and all connected partitions into at most $k$ parts exactly on triangular cacti. For other connected graphs, one may ask for a weaker bound on the number of parts: by Proposition~\ref{prop:cycles}, every cycle satisfies the conclusion under the bound $k-1$. Does a bound of the form $k-g(G)$, depending on the cyclomatic number or on another parameter of $G$, suffice in general? Analogous questions for approximate connected MMS allocations of chores on graphs with cycles are open as well; for three agents on cycles, $7/6$-approximate connected MMS allocations of chores were obtained in~\cite{XiaoEtAl2023arXiv}.
\end{problem}

\begin{problem}[Unequal entitlements]
Theorem~\ref{thm:A} concerns $k$ partitions, each with at most $k$ parts, for which $\sum_i1/|\mathcal P_i|\ge1$. The latter condition does not suffice in general: for the path with colour word $3231323$, the partitions $\mathcal P_c=\Comp(T-E_c)$ have $2$, $3$, and $5$ parts for $c=1,2,3$, so that $\frac12+\frac13+\frac15=\frac{31}{30}>1$, but every ordering of $\{1,2,3\}$ is a subsequence of the word, and by Proposition~\ref{prop:paths} there is no covering by one part of each partition. Which conditions on the numbers $|\mathcal P_i|$ guarantee a rainbow partition? Such conditions would bear on connected versions of weighted maximin shares for chores.
\end{problem}

\noindent\textbf{AI disclosure.} The project was initiated by Lonc, who presented to the remaining authors the conjecture that Theorem~\ref{thm:A} should hold. Then the authors worked (mostly independently) and obtained several partial results supporting the conjecture. The final steps of the proof were aided by generative AI. The authors take full responsibility for the contents of the paper.

\bibliographystyle{amsplain}
\bibliography{references}

\end{document}